\documentclass[reqno,11pt]{article}
\usepackage{a4wide,xcolor,eucal,enumerate,mathrsfs}
\usepackage[normalem]{ulem}
\usepackage[pdfborder={0 0 0}]{hyperref}

\usepackage{amsmath,amssymb,epsfig,amsthm}%bbm,
\usepackage[latin1]{inputenc}

\newcommand{\C}{\mathbb{C}}

\newcommand{\N}{\mathbb{N}}

\newcommand{\R}{\mathbb{R}}

\newcommand{\mm}{{\mbox{\boldmath$m$}}}
\newcommand{\nn}{{\mbox{\boldmath$n$}}}

\newcommand{\snn}{{\mbox{\scriptsize\boldmath$n$}}}

\newcommand{\ggamma}{{\mbox{\boldmath$\gamma$}}}

\newcommand{\ppi}{{\mbox{\boldmath$\pi$}}}

\newcommand{\sfd}{{\sf d}}

\newcommand{\sfS}{{\sf S}}

\newcommand{\rme}{{\mathrm e}}
\newcommand{\Kliminf}{K\kern-3pt-\kern-2pt\mathop{\rm lim\,inf}\limits}  % Kuratowski liminf di insiemi
\newcommand{\supp}{\mathop{\rm supp}\nolimits}   % supporto 
\newcommand{\diam}{\mathop{\rm diam}\nolimits}   % diametro
\newcommand{\Lip}{\mathop{\rm Lip}\nolimits}          %Lipschitz constant
\renewcommand{\d}{{\mathrm d}}
\newcommand{\dt}{{\d t}}

\newcommand{\restr}[1]{\lower3pt\hbox{$|_{#1}$}}

\newcommand{\down}{\downarrow}              %frecce in su e in giu nei limiti
\newcommand{\up}{\uparrow}
\newcommand{\eps}{\varepsilon}  
\newcommand{\nchi}{{\raise.3ex\hbox{$\chi$}}}

\newcommand{\forevery}{\text{for every }}
\newcommand{\Pc}[2]{\overline{#1}\kern-2pt^{\vphantom 0}_{#2}}
\newcommand{\Probabilities}[1]{\mathscr P(#1)}          % misure di probabilita'
\newcommand{\ProbabilitiesTwo}[1]{\mathscr P_2(#1)}     % misure di probabilita' con momento quadratico finito
\numberwithin{equation}{section}

\newtheorem{theorem}{Theorem}[section]

\newtheorem{lemma}[theorem]{Lemma}
\newtheorem{proposition}[theorem]{Proposition}

\newtheorem{definition}[theorem]{Definition}

\newtheorem{remark}[theorem]{Remark}

\renewcommand{\mm}{\mathfrak m}
\renewcommand{\nn}{\mathfrak n}

\newcommand{\ent}[1]{{\rm Ent}_{\mm}(#1)}              % entropia. Ha come argomento la misura di cui si calcola l'entropia (quella di riferimento e' sempre $\mm$)
\newcommand{\entv}{{\rm Ent}_{\mm}}                    % entropia senza argomento. La uso principalmente quando scrivo la slope
\newcommand{\prob}{\Probabilities}
\newcommand{\probt}{\ProbabilitiesTwo}
\newcommand{\geo}{{\rm{Geo}}}                       % insieme delle geodetiche
\newcommand{\e}{{\rm{e}}}                           % mappa di valutazione, bisogna mettere `a mano' il tempo t
\newcommand{\gopt}{{\rm{OptGeo}}}                   % displacement plans ottimali
\newcommand{\fr}{\hfill$\blacksquare$}                      %quadratino nero alla fine del remark, se non vi piace, la cosa migliore e' `svuotare' la macro, cosi' non bisogna intervenire sul testo

\newcommand{\weakgrad}[1]{|D #1|_w}                % gradient ottenuto integrando lungo "quasi tutte" le geodetiche

\renewcommand{\C}{{\sf Ch}_*}

\newcommand{\Wgh}{V}

\newcommand{\heatl}{{\sf h}}
\newcommand{\heatw}{{\mathscr H}}

\newcommand{\calW}{{\mathcal S}^2}
\newcommand{\Curvesnonpara}[1]{\mathscr C(#1)}
\newcommand{\res}{\mathop{\hbox{\vrule height 7pt width .5pt depth 0pt
\vrule height .5pt width 6pt depth 0pt}}\nolimits} % macro per la restrizione
\newcommand{\Goodmeasures}{{\mathcal M}}
\renewcommand{\C}{{\sf Ch}}
\newcommand{\Gbil}[2]{G(#1,#2)}   

\newcommand{\ke}[2]{\heatw_{#2}(\delta_{#1})}
\newcommand{\ked}[2]{\rho_{#2}[#1]}

\title{Riemannian Ricci curvature lower bounds in metric \\measure spaces with $\sigma$-finite measure}
\begin{document}

\author{Luigi Ambrosio\
   \thanks{Scuola Normale Superiore, Pisa, \textsf{l.ambrosio@sns.it}}
   \and
   Nicola Gigli\
   \thanks{University of Nice, \textsf{nicola.gigli@unice.fr}}
   \and
   Andrea Mondino
   \thanks{Scuola Normale Superiore, Pisa, \textsf{andrea.mondino@sns.it}}
   \and
   Tapio Rajala
   \thanks{University of Jyv\"askyl\"a, \textsf{tapio.m.rajala@jyu.fi}}
   }

\maketitle

\begin{abstract}
In prior work \cite{Ambrosio-Gigli-Savare11b} of the first two authors with Savar\'e, a new Riemannian notion of lower bound for Ricci curvature 
in the class of metric measure spaces $(X,\sfd,\mm)$ was introduced, and the corresponding class of spaces denoted by
$RCD(K,\infty)$. This notion relates the $CD(K,N)$ theory of Sturm and Lott-Villani, in the case $N=\infty$, to the Bakry-Emery approach. In
\cite{Ambrosio-Gigli-Savare11b} the $RCD(K,\infty)$ property is defined in three equivalent ways and several properties of 
$RCD(K,\infty)$ spaces, including the regularization properties of the heat flow,
the connections with the theory of Dirichlet forms and the stability under tensor products, are provided. In \cite{Ambrosio-Gigli-Savare11b}
only finite reference measures $\mm$ have been considered. The goal of this paper is twofold: on one side we extend these results to general $\sigma$-finite spaces, on the other we remove a technical assumption appeared in  \cite{Ambrosio-Gigli-Savare11b} concerning a strengthening of the $CD(K,\infty)$ condition. This more general class of spaces
includes Euclidean spaces endowed with Lebesgue measure, complete noncompact Riemannian manifolds with bounded geometry and the
pointed metric measure limits of manifolds with lower Ricci curvature bounds. 
\end{abstract}

\tableofcontents

\section{Introduction}

In a recent paper \cite{Ambrosio-Gigli-Savare11b} written jointly with Savar\'e, the first and second author introduced a notion of Riemannian Ricci lower
bound for metric measure spaces $(X,\sfd,\mm)$, relying on the calculus tools they had developed in \cite{Ambrosio-Gigli-Savare11}. 
This definition, in the spirit of the $CD(K,N)$ theory proposed by Lott-Villani \cite{Lott-Villani09} and
Sturm \cite{Sturm06I,Sturm06II} relies on optimal transportation tools and suitable convexity properties of the relative entropy functional
${\rm Ent}_\mm$. In the framework of \cite{Ambrosio-Gigli-Savare11b}, these conditions are enforced
adding the assumption that the so-called Cheeger energy (playing here the role of the classical Dirichlet energy) is quadratic.\\
More precisely, the class of $RCD(K,\infty)$ spaces of \cite{Ambrosio-Gigli-Savare11b} can be defined in 3 equivalent
ways thanks to this equivalence result (see $\S\ref{ssevi}$ for the precise formulation of gradient flows involved here, in 
the metric sense and in the $EVI_K$ sense):

\begin{theorem}\label{thm:main}
 {\rm \cite{Ambrosio-Gigli-Savare11b} }Let $(X,\sfd,\mm)$ be a metric measure space with $(X,\sfd)$ complete and separable, $\mm(X)\in (0,\infty)$
 and $\supp\mm=X$. Then the following 
 are equivalent.
 \begin{enumerate}
  \item[(i)] $(X,\sfd,\mm)$ is a strict $CD(K,\infty)$ space and the $W_2$-gradient flow $\heatw_t$ of $\entv$ on $\ProbabilitiesTwo{X}$ is additive.
  \item[(ii)] $(X,\sfd,\mm)$ is a strict $CD(K,\infty)$ space and $\C$ is a quadratic form on $L^2(X,\mm)$.
  \item[(iii)] $(X,\sfd,\mm)$ is a length space and any $\mu \in \ProbabilitiesTwo{X}$ is the starting point of an $EVI_K$
                gradient flow of $\entv$.
 \end{enumerate}
\end{theorem}

This equivalence is crucial for the study of the spaces $RCD(K,\infty)$: for instance the fine properties of the heat
flow and the Bakry-Emery condition obtained in \cite{Ambrosio-Gigli-Savare11b} need (ii), while stability of $RCD(K,\infty)$ spaces under
Sturm's convergence \cite{Sturm06II} of metric measure spaces (a variant of measured Gromov-Hausdorff convergence)
depends in a crucial way on (iii) and on the stability properties of $EVI_K$ flows of \cite{Ambrosio-Gigli-Savare08}.

The aim of this paper is the extension of the theory of $RCD(K,\infty)$ spaces to a class of $\sigma$-finite metric measure
spaces. This extension includes fundamental examples such as the Lebesgue measure in $\R^n$,
noncompact Riemannian manifolds with bounded geometry and the pointed metric measure limits of manifolds with
lower Ricci curvature bounds studied by Cheeger and Colding \cite{Cheeger-Colding97I, Cheeger-Colding97II, Cheeger-Colding97III}. In our class of spaces  
we obtain the perfect analogue of Theorem~\ref{thm:main} (see Theorem~\ref{thm:main1}).
Actually, even in the finite case we improve Theorem~\ref{thm:main}, replacing strict $CD(K,\infty)$ with $CD(K,\infty)$ in (i)
and (ii): this is possible mainly thanks to the fine results of Section~\ref{sec:Tapio}.
%The measures $\mm$ we will be studying satisfy the quantitative $\sigma$-finiteness condition
%\begin{equation}\label{eq:stras}
 %\int_X\rme^{-{\sf c}\sfd^2(x_0,x)}\,\d\mm(x) < \infty
%\end{equation} 
%for some $x_0\in X$ and ${\sf c}\in (0,\infty)$. As illustrated in \cite[Remark~4.21]{Ambrosio-Gigli-Savare11} this condition is already needed 
%and close to being sharp for stochastic completeness (i.e. mass conservation for the heat flow) and it is also a consequence of the
%$CD(K,\infty)$ condition as formulated by Sturm in \cite{Sturm06I}, so it is very natural within this theory.

Let us now briefly and informally explain the terminology implicit in Theorem~\ref{thm:main} and the technical difficulties arising when one considers
$\sigma$-finite reference measures $\mm$. Cheeger's energy $\C$ can be defined in $L^2(X,\mm)$ by a relaxation procedure
$$
\C(f):=\frac 12\inf\left\{\liminf_{h\to\infty}\int_X|Df_h|^2\,\d\mm:\ \text{$f_h$ Lipschitz, $f_h\to f$ in $L^2(X,\mm)$}\right\},
$$
where $|Df|$ is the slope, see \eqref{eq:slope}. Instead of this direct construction, 
we shall exclusively work in this paper with another equivalent one (equivalence follows by
Theorem~6.2 of \cite{Ambrosio-Gigli-Savare11}), based on the notion of weak upper gradient $\weakgrad{f}$, see Definition~\ref{def:wug}.
The weak upper gradient provides integral representation for $\C$, namely 
$$
\C(f)=\frac12\int_X\weakgrad{f}^2\,\d\mm\qquad\text{whenever $\C(f)<\infty$}.
$$ 
Since $\C$ is convex and lower semicontinuous on $L^2(X,\mm)$, its gradient flow ${\sf h}_tf$ is well defined starting from any initial
condition. One of the main results of \cite{Ambrosio-Gigli-Savare11} is the coincidence of ${\sf h}_t $ with the quadratic optimal transport distance semigroup 
$\heatw_t$ (the $W_2$ gradient flow of ${\rm Ent}_\mm$) under the $CD(K,\infty)$ assumption: more precisely, if $f\in L^2(X,\mm)$ and 
$\int f(x)\sfd^2(x,x_0)\,\d\mm(x)$ is finite, then $\heatw_t(f\mm)=({\sf h}_t f)\mm$, see Theorem~\ref{thm:heatgf}. 
This explains the connection between (i) and (ii), where finiteness of $\mm$ does not play any role. Passing to
the $EVI_K$ condition, deeply studied by the first two authors and Savar\'e in \cite{Ambrosio-Gigli-Savare08} and by
Daneri and Savar\'e in \cite{Daneri-Savare08}, it amounts (see Definition~\ref{def:EVIK}) to a family of differential inequalities indexed
by $\sigma\in\probt{X}$:
\begin{equation}\label{eq:EVI}
\frac{\d}{\d t}\frac 12W_2^2(\mu_t,\sigma)\leq {\rm Ent}_\mm(\sigma)-{\rm Ent}_\mm(\mu_t)-\frac{K}{2}W_2^2(\mu_t,\sigma)
\quad\text{for a.e. $t\in (0,\infty)$.}
\end{equation}
Set $\mu_t=({\sf h}_t f)\mm$ and let $\varphi_t$ be Kantorovich potentials from $\mu_t$ to $\sigma$.
The analysis in \cite{Ambrosio-Gigli-Savare11b} shows that
\begin{equation}\label{eq:florence1}
\frac{\d}{\dt}\frac 12 W_2^2(\mu_t,\sigma)\leq \lim_{\eps\downarrow 0}\frac{\C(f_t-\eps\varphi_t)-\C(f_t)}{\eps}
\end{equation}
on the one hand, and that the $CD(K,\infty)$ condition gives
\begin{equation}\label{eq:florence2}
\lim_{\eps\downarrow 0}\frac{\C(\varphi_t-\eps f_t)-\C(\varphi_t)}{\eps}
\leq {\rm Ent}_\mm(\sigma)-{\rm Ent}(\mu_t)-\frac{K}{2}W_2^2(\mu_t,\sigma)
\end{equation}
on the other hand.
If $\C$ is quadratic, then we can formally write that both the right hand side in \eqref{eq:florence1} and the left hand
side in \eqref{eq:florence2} coincide with $-\int_X D f_t\cdot D\varphi_t\,\d\mm$, thus providing the connection from
(ii) to (iii). However, in the derivation of
\eqref{eq:florence2} a key role is played by the Sobolev regularity of $\log f_t$, that can be easily achieved if
$f_t\geq c>0$. But, this assumption is not compatible with the $\sigma$-finite case, since $f_t$ is a probability
density, and even local space-time lower bounds on $f_t$ can hardly be obtained in our framework, where no
finite dimensionality assumption on $(X,\sfd,\mm)$ is made. It turns out that this derivation is still possible, but
only working in a time-dependent weighted Sobolev space: formally we write 
$$
\int_X Df_t\cdot D\varphi_t\,\d\mm=\int_X D\log f_t\cdot D\varphi_t\,\d(f_t\mm)
$$
and, thanks to the energy dissipation estimate
$$
{\rm Ent}_\mm(f_T\mm)+\int_0^T\int_X\frac{\weakgrad{f_t}^2}{f_t}\,\d\mm\,\d t\leq{\rm Ent}_\mm(f\mm),
$$
we know that $\log f_t$ belongs for a.e. $t$ to the Sobolev space with weight $f_t$. Then we prove that for a.e. $t>0$
the first inequality \eqref{eq:florence1} holds, when written in terms of weighted Sobolev spaces, for any choice of the Kantorovich potential 
$\varphi_t$, while the second inequality \eqref{eq:florence2} holds for at least one. This suffices for the derivation of \eqref{eq:EVI}.

Besides the application to $\sigma$-finite $RCD(K,\infty)$ spaces, several results of this paper have an independent
interest and do not rely on curvature assumptions: see, for instance, Lemma~\ref{lem:GammaConvKant} which provides compactness properties 
of Kantorovich potentials and Theorem~\ref{thm:change}, which analyzes the weighted Cheeger energies. 
Also, it is worthwhile to mention
that existence of geodesics with $L^\infty$ bounds of Section~\ref{sec:Tapio} applies to $\sigma$-finite $CD(K,\infty)$ spaces, 
i.e. no quadratic assumption
on $\C$ is needed for the results of the section. Also, since finiteness of $\mm$ was used in \cite{Ambrosio-Gigli-Savare11b}
essentially only for the equivalence of Theorem~\ref{thm:main}, we describe in the last section the properties of $RCD(K,\infty)$ spaces 
proved in \cite{Ambrosio-Gigli-Savare11b}, whose proof extends with no additional effort to the $\sigma$-finite case: among them we just mention the
Bakry-Emery condition
$$
\weakgrad {({\sf h}_t f)}^2\leq\rme^{-2Kt}\weakgrad{f}^2\qquad\text{$\mm$-a.e. in $X$.}
$$
Further analysis of the Bakry-Emery condition will appear in the forthcoming paper \cite{AGSBaEm}.
The extension of the stability of the $RCD(K,\infty)$ condition under Sturm's metric measure convergence to
the $\sigma$-finite case is far from being trivial. We refer to  \cite{AmbrosioGigliMondinoSavare} for the positive answer to this question.

The paper is organized as follows. In Section~\ref{sec:preliminaries} we gather a few facts on relative entropy and optimal
transportation, mostly stated without proofs (standard references are \cite{Ambrosio-Gigli11}, \cite{Ambrosio-Gigli-Savare08},
\cite{Villani09}); the only original contribution is a compactness result for Kantorovich potentials via De Giorgi's $\Gamma$-convergence stated in
Lemma~\ref{lem:GammaConvKant}. 

In Section~\ref{sec:Cheeger} we recall the main results of the theory of weak gradients
as developed by the first two authors with Savar\'e in \cite{Ambrosio-Gigli-Savare11}, emphasizing also the connections with the points of view 
developed by Cheeger in \cite{Cheeger00}, Koskela-MacManus in \cite{Koskela-MacManus} and 
Shanmugalingam in \cite{Shanmugalingam00}. The main result of the section is Theorem~\ref{thm:change} which states that, 
for probability densities $\rho=g\mm$ with $g\in L^\infty(X,\mm)$ and $\C(\sqrt{g})<\infty$, roughly speaking weak gradients w.r.t to $\mm$ and
weak gradients with respect to $\rho$ are the same, even though no (local) lower bound on $g$ is assumed. Furthermore,
Cheeger's energy ${\C}_\rho$ induced by $\rho$ is quadratic if $\C$ is quadratic. Section~\ref{sec:Tapio} is crucial for the development
of (short time) $L^\infty$ estimates for displacement interpolation in $CD(K,\infty)$ spaces (see Theorem~\ref{thm:goodgeodesics} for a
precise statement) which are new in the situation when 
$(X,\sfd)$ is unbounded and $\mm$ is not finite. These estimates,
which hold when the density of the first measure decays at least as $c_1\rme^{-c_2\sfd^2(x,x_0)}$ for some $c_1,\,c_2>0$ and the second measure has 
bounded density and support, are obtained combining carefully entropy minimization (an approach proposed by Sturm and
then developed  by Rajala in \cite{R2011b,R2012}) and splitting of optimal geodesic plans. Section~\ref{sec:auxiliary}
is devoted to the proof of some auxiliary convergence results dealing with entropy, difference quotients of probability densities and Kantorovich
potentials, bilinear form ${\C}_\rho$ associated to a measure $\rho\in\probt{X}$ as in Section~\ref{sec:Cheeger}. Section~\ref{sec:lastq}
contains the proof of Theorem~\ref{thm:main1}, which provides the equivalence result analogous to Theorem~\ref{thm:main} 
in the present $\sigma$-finite setting.

\smallskip
\noindent {\bf Acknowledgement.} The authors warmly thank Giuseppe Savar\'e for his detailed and helpful comments
on a preliminary version of this paper and the reviewer for his constructive comments. The authors acknowledge the support
of the ERC ADG GeMeThNES. T.R. acknowledges the support of the Academy of Finland, project
no. 137528.

\section{Preliminaries}\label{sec:preliminaries}

In this section we introduce our notation, including the relative entropy functional ${\rm Ent}_\nn$ in \eqref{eq:defRelentropy},
the slope $|Df|$ of a function $f$ in \eqref{eq:slope}, the one-sided slopes $|D^\pm f|$ in
\eqref{eq:onesidedslopes}, the class $AC^p(J;X)$ of absolutely continuous curves with metric derivative in $L^p(J)$, 
the class of geodesics \eqref{defgeo} and the notions of geodesic and length space. We then review optimal transport,
prove the existence of special Lipschitz Kantorovich potentials (Proposition~\ref{prop:goodKant}) and prove a compactness theorem of
Kantorovich potentials (Lemma~\ref{lem:GammaConvKant}).

We assume throughout the paper that $(X,\sfd,\mm)$ is a metric measure space with $(X,\sfd)$ complete and separable  and  $\mm$ being a nonnegative 
Borel measure finite on bounded sets and satisfying $\supp\mm=X$.
% and \eqref{eq:stras} for some ${\sf c}> 0$ and $x_0\in X$. 
%This assumption includes finite measures and large classes of measures finite on bounded sets such as the Lebesgue measure or the
%Riemannian volume measure of manifolds with bounded geometry. 

We denote by $\prob X$ the space of Borel probability measures on $(X,\sfd)$ and
set \[
\probt X :=\Big\{\mu\in\prob X\ :\ \int_X\sfd^2(x_0,x)\,\d\mu(x)<\infty\,\,\,\text{for some (and hence all) $x_0\in X$}\Big\}.
\]
%According to \eqref{eq:stras} we denote
%$z=\int_X\rme^{-c\sfd^2(x,x_0)}\,\d\mm$ and
%\begin{equation}\label{eq:grygorian1}
%\tilde\mm=\frac 1z \rme^{-c\sfd^2(x,x_0)}\mm\in\Probabilities{X},
%\quad\qquad\Wgh(x)=\sfd(x,x_0).
%\end{equation}
Given a nonnegative Borel measure $\nn$, the \emph{relative entropy functional} ${\rm Ent}_\nn:\probt{X}\to [-\infty,\infty]$ 
with respect to $\nn$ is defined as in Sturm's paper \cite{Sturm06I} by
\begin{equation}\label{eq:defRelentropy}
{\rm Ent}_\nn(\mu):=
\begin{cases}
\lim\limits_{\epsilon\downarrow 0}\int_{\{\rho>\epsilon\}} \rho\log \rho\,\d\nn& \text{if $\mu=\rho\nn$}; \\
\infty & \text{otherwise.}
\end{cases}
\end{equation}
It coincides with $\int_{\{\rho>0\}}\rho\log\rho\,\d\nn\in [-\infty,\infty)$ if the positive part of $\rho\log\rho$ is $\nn$-integrable,
and it is equal to $\infty$ otherwise.

In the sequel we use the notation
\begin{equation}\label{def:DDDD}
D({\rm Ent}_\nn):=\left\{\mu\in\probt{X}:\ {\rm Ent}_\nn(\mu)\in [-\infty,\infty)\right\}.
\end{equation}

By Jensen's inequality, ${\rm Ent}_\nn$ is nonnegative when $\nn\in\prob{X}$. More generally,
we recall (see \cite[Lemma~7.2]{Ambrosio-Gigli-Savare11} for the simple proof) that when $\nn$ satisfies the growth condition
\begin{equation}\label{eq:growthcond}
\int_X\rme^{-{\sf c}\sfd^2(x_0,x)}\,\d\nn(x) < \infty,
\end{equation} 
for some $x_0\in X$ and ${\sf c}\in (0,\infty)$, then ${\rm Ent}_\nn$ can bounded from below as follows. Letting $z=\int_X\rme^{-c\sfd^2(x,x_0)}\,\d\nn$ and 
\begin{equation}\label{eq:grygorian2}
\tilde\nn=\frac 1z \rme^{-c\sfd^2(x,x_0)}\nn\in\Probabilities{X},
\quad\qquad\Wgh(x)=\sfd(x,x_0),
\end{equation}
and using the simple formula for the change of the reference measure
\begin{equation}\label{eq:changeentropy}
{\rm Ent}_\nn(\mu)={\rm Ent}_{\tilde\nn}(\mu)-c\int_X\Wgh^2\,\d\mu-\log z,\qquad\forall\mu\in\probt X,
\end{equation}
we see that ${\rm Ent}_\nn$ can be bounded from below in terms of the second moment of $\mu$. It is important to recall that if $(X,\sfd,\mm)$ is a $CD(K,\infty)$ space (see Definition~\ref{def:CD}), then the reference measure $\mm$ always satisfies the growth condition \eqref{eq:growthcond}, 
as shown by Sturm in \cite[Theorem 4.24]{Sturm06I}.

\subsection{Metric structure}

We shall denote by ${\rm Lip}(X)$ the space
of Lipschitz functions $f:X\to\R$ and by ${\rm Lip}_b(X)$ the subspace of bounded
Lipschitz functions.

Given $f:X\to\R$ we define its slope $|Df|$ at $x$ by
\begin{equation}\label{eq:slope}
|Df|(x):=\limsup_{y\to x}\frac{|f(y)-f(x)|}{\sfd(y,x)}.
\end{equation}
We shall also use, in connection with Kantorovich potentials,
the one-sided counterparts of the slope, namely the ascending slope and descending slopes:
\begin{equation}\label{eq:onesidedslopes}
|D^+f|(x):=\limsup_{y\to x}\frac{[f(y)-f(x)]^+}{\sfd(y,x)},\qquad
|D^-f|(x):=\limsup_{y\to x}\frac{[f(y)-f(x)]^-}{\sfd(y,x)}.
\end{equation}

Given an open interval $J\subset\R$, an
exponent $p\in [1,\infty]$ and $\gamma:J\to X$, we say that
$\gamma$ belongs to $AC^p(J;X)$ if there exists $g\in L^p(J)$ satisfying
$$
\sfd(\gamma_s,\gamma_t)\leq\int_s^t g(r)\,\d r\qquad\forall s,\,t\in
J,\,\,s<t.
$$
The case $p=1$ corresponds to \emph{absolutely continuous} curves, denoted $AC(J;X)$. 
It turns out that, if $\gamma$ belongs to $AC^p(J;X)$, there is a minimal function $g$ with this property, called
\emph{metric derivative} and given for a.e. $t\in J$ by
$$
|\dot\gamma_t|:=\lim_{s\to t}\frac{\sfd(\gamma_s,\gamma_t)}{|s-t|}.
$$
See \cite[Theorem~1.1.2]{Ambrosio-Gigli-Savare08} for the simple
proof. We say that an absolutely continuous curve $\gamma_t$ has
\emph{constant speed} if $|\dot\gamma_t|$ is (equivalent to) a
constant.

We call $(X,\sfd)$ \emph{a geodesic space} if for any $x_0,\,x_1\in
X$ there exists 
$\gamma:[0,1]\to X$ satisfying $\gamma_0=x_0$, $\gamma_1=x_1$ and
\begin{equation}\label{defgeo}
\sfd(\gamma_s,\gamma_t)=|t-s|\sfd(\gamma_0,\gamma_1)\qquad\forall s,\,t\in
[0,1].
\end{equation}
We will denote by $\geo(X)$ the space of all constant speed
geodesics $\gamma:[0,1]\to X$, namely $\gamma\in\geo(X)$ if
\eqref{defgeo} holds. Recall also that the weaker notion of \emph{length} space: 
for all $x_0,\,x_1\in X$ and $\eps>0$ there exists $\gamma\in AC([0,1];X)$ such
that $\int_0^1|\dot\gamma_t|\,\d t<\sfd(x_0,x_1)+\eps$.

{F}rom the measure-theoretic point of view, when considering measures
on $AC^p(J;X)$ (resp. $\geo(X)$), we shall consider
them as measures on the Polish space $C(J;X)$ endowed with the sup
norm, concentrated on the Borel set $AC^p(J;X)$ (resp. closed set $\geo(X)$). We shall also use the
notation $\e_t:C(J;X)\to X$, $t\in J$, for the evaluation map at time $t$, namely
$\e_t(\gamma):=\gamma_t$.

\subsection{Optimal transport}

Given $\mu,\,\nu\in\probt X$, we define the quadratic optimal transport distance $W_2$
between them as
\begin{equation}\label{eq:W2def}
W_2^2(\mu,\nu):=\inf\int_{X\times X} \sfd^2(x,y)\,\d\ggamma(x,y),
\end{equation}
where the infimum is taken among all Kantorovich transport plans, namely 
probability measures $\ggamma$ on $X\times X$
such that
\[
\pi^1_\sharp\ggamma=\mu,\qquad\pi^2_\sharp\ggamma=\nu.
\]
Here, for $\mu\in \prob X$, a topological space $Y$ and a $\mu$-measurable
map $T:X\to Y$, the push-forward measure $T_\sharp \mu\in \prob Y$ is defined by
$T_\sharp\mu(B):=\mu(T^{-1}(B))$ for every Borel set $B\subset Y$.

Since $(X,\sfd)$ is complete and separable,
the space $(\probt X,W_2)$ is complete and separable. Since the cost $\sfd^2$ is lower semicontinuous,
the infimum in the definition \eqref{eq:W2def} of $W_2^2$ is attained.
All plans $\ggamma$ achieving the minimum will be called optimal.

For all $\mu,\,\nu\in\probt{X}$ Kantorovich's duality formula holds:
\begin{equation}
\label{eq:dualitabase} \frac12 W_2^2(\mu,\nu)=\sup
\left\{\int_X\varphi\,\d\mu+\int_X\psi\,\d\nu:\
\varphi(x)+\psi(y)\leq \frac12 \sfd^2(x,y)\right\},
\end{equation}
where the supremum is taken among all functions $\varphi\in L^1(X,\mu)$ and $\psi\in L^1(X,\nu)$.

Recall that the \emph{$c$-transform} $\varphi^c$
of $\varphi:X\to\R\cup\{-\infty\}$ is defined by
\[
\varphi^c(y):=\inf\left\{\frac{\sfd^2(x,y)}2-\varphi(x):\ x\in X\right\}
\]
and that $\psi$ is said to be \emph{$c$-concave} if
$\psi=\varphi^c$ for some $\varphi$.

\begin{definition}[Kantorovich potential]\label{def:Kant}
We say that a map $\varphi:X\to\R\cup\{-\infty\}$ is a Kantorovich potential relative to $(\mu,\nu)$ if:
\begin{itemize}
\item[(i)] there exists a  Borel map $\psi:X\to\R\cup\{-\infty\}$ such that $\psi\in L^1(X,\nu)$ and $\varphi=\psi^c$;
\item[(ii)] $\varphi\in L^1(X,\mu)$ and the pair $(\varphi,\psi)$ maximizes \eqref{eq:dualitabase}.
\end{itemize}
\end{definition}
Notice that the inequality $\varphi(x)+\psi(y)\leq\tfrac12\sfd^2(x,y)$, when integrated against an optimal plan
$\ggamma$, forces the integrability of the positive part of $\varphi$. For this reason, in (ii) we may equivalently
require integrability of the negative part of $\varphi$ only.  In the next proposition we illustrate
some key properties of Kantorovich potentials $\varphi$ and show how, in the special case when $\supp\nu$ is bounded,
a special choice of $\psi$ provides better properties of $\varphi=\psi^c$.

\begin{proposition}[Existence of Kantorovich potentials]\label{prop:goodKant}
If $\mu,\,\nu\in\probt{X}$, then a Kantorovich potential $\varphi=\psi^c$ relative to $(\mu,\nu)$
exists and satisfies 
\begin{equation}\label{eq:itforza1}
\varphi(x)+\psi(y)=\tfrac 12\sfd^2(x,y)\qquad\text{for $\ggamma$-a.e. in $(x,y)\in X\times X$}
\end{equation}
for any optimal Kantorovich plan $\ggamma$ and 
\begin{equation}\label{eq:itforza}
|D^+\varphi|(x)\leq\sfd(x,y)\qquad\text{for $\ggamma$-a.e. $(x,y)$.}
\end{equation}
In addition, if $\supp\nu\subset\overline{B}_R(y_0)$ for some $R\geq 1$, then a locally Lipschitz
Kantorovich potential $\varphi=\psi^c$ exists
with $\psi\equiv -\infty$ on $X\setminus\supp\nu$, $\psi\leq R^2/2$ on $\supp\nu$ and
\begin{equation}\label{GrowthKant}
|D \varphi| (x) \leq R+\sfd(x,y_0), \qquad |\varphi(x)| \leq 2R^2 (1+\sfd^2(x,y_0)).
\end{equation} 
\end{proposition}
\begin{proof} Since any complete and separable metric space can be isometrically embedded in a complete, separable and
geodesic metric space we can assume with no loss of generality that the space $(X,\sfd)$ is geodesic.
The existence part is well known, so let us discuss briefly \eqref{eq:itforza}, the choice of gauge and
the regularity properties of $\varphi$ when $\nu$ has bounded support. From \eqref{eq:itforza1} and
the inequality $\varphi+\varphi^c\leq\sfd^2/2$ we get
$$
\varphi(z)-\varphi(x)\leq\frac{1}{2}\bigl(\sfd^2(z,y)-\sfd^2(x,y)\bigr) \qquad\text{for all $z$}
$$ 
for $\ggamma$-a.e. $(x,y)$, so that $|D^+\varphi|(x)\leq\sfd(x,y)$ for $\gamma$-a.e. $(x,y)$.

Now, let us set
$$
\tilde\psi(x):=\begin{cases}
\psi(x) &\text{if $x\in\supp\nu$};\\
-\infty &\text{otherwise,}
\end{cases}
$$ 
and $\tilde{\varphi}:=(\tilde\psi)^c$.
Since $\tilde\varphi\geq\varphi$, it is obvious that its negative part is
$\mu$-integrable and that $(\tilde\varphi,\tilde\psi)$ is a maximizing pair, so that $\tilde\varphi$ is a Kantorovich 
potential. From
$$
\tilde\varphi(x)=\inf_{y\in\supp\nu}\frac 12\sfd^2(x,y)-\tilde\psi(y)
$$
and the inclusion $\supp\nu\subset B_R(y_0)$ it is immediate to obtain the linear growth of $|D\tilde\varphi|$, in the form stated in \eqref{GrowthKant}.
Finally, possibly adding and subtracting the same constant to the potentials in the maximizing pair, we can assume
that $\tilde\varphi(y_0)=0$. Then, the inequality $\tilde\psi\leq \frac 12\sfd^2(y_0,\cdot)$ gives
$\tilde\psi\leq R^2/2$ on $\supp\nu$. The linear growth of $|D\tilde\varphi|$ gives the quadratic growth
of $|\varphi|$, since $(X,\sfd)$ is geodesic. \end{proof}

In the proof of the next lemma we use De Giorgi's $\Gamma$-convergence. Strictly speaking, we use $\Gamma^-$-convergence, 
the one designed for convergence of minimum problems. We recall the definition and the basic facts, referring to
Dal Maso's book \cite{DalMaso} for a full account of this theory. If $(Y,d)$ is a metric space
and $f_h:Y\to [-\infty,+\infty]$, $f:Y\to [-\infty,+\infty]$ are lower semicontinuous, we say that $(f_h)$ $\Gamma$-converges
to $f$ and write $f=\Gamma-\lim_h f_h$ if:
\begin{itemize}
\item[(a)] for any sequence $(y_h)\subset Y$ convergent to $y\in Y$, one has $\liminf_h f_h(y_h)\geq f(y)$;
\item[(b)] for all $y\in Y$ there exists $(y_h)\subset Y$ convergent to $y$ and satisfying $\limsup_h f_h(y_h)\leq f(y)$.
\end{itemize}
It is immediate to check that $\Gamma$-convergence is invariant by additive constant perturbations. In addition, (a)
yields that $f\mapsto\inf_A f$ is lower semicontinuous w.r.t. $\Gamma$-convergence for any open set $A\subset Y$, while
(b) yields that $f\mapsto\min_K f$ is upper semicontinuous w.r.t. $\Gamma$-convergence for any compact set $K\subset Y$.
If $Y$ is compact we can choose $A=K=Y$ to obtain 
\begin{equation}\label{eq:vuoleilreferee}
\Gamma-\lim_{h\to\infty}f_h=f\qquad\Longrightarrow\qquad
\lim_{h\to\infty} \min_Y f_h=\min_Y f.
\end{equation}
We need one more property of $\Gamma$-convergence: if $Y$ is separable, then any sequence of lower semicontinuous
maps $f_h:Y\to [-\infty,+\infty]$ admits a $\Gamma$-convergent subsequence $f_{h(k)}$. To see this, let $\mathscr U$ be a
countable basis of open sets of $Y$ and extract with a diagonal argument a subsequence $h(k)$ such that $\inf_U f_{k(k)}$ has a limit in
$[-\infty,+\infty]$ for all $U\in\mathscr U$. Then, the function
$$
f(y):=\sup_{U\ni y,\,U\in{\mathscr U}}\lim_{k\to\infty}\inf_U f_{h(k)}\qquad y\in Y
$$
provides the $\Gamma$-limit of $f_{h(k)}$.

\begin{lemma}[Compactness of Kantorovich potentials]\label{lem:GammaConvKant}
Consider probability densities $\sigma,\,\eta=f\mm,\,\eta_n=f_n\mm\in\probt{X}$ satisfying the following conditions:
\begin{itemize}
\item[(a)] $\sigma$ has compact support;
\item[(b)] $f_n \to f$ $\mm$-a.e. in $X$ and 
$\sup_n f_n (x) (1+\sfd^2(x, x_0)) \in L^1(X,\mm)$ for some $x_0 \in X$.
\end{itemize}
Suppose there exist $C>0$ and Kantorovich potentials $\varphi_n=\psi_n^c$ relative to $(\eta_n,\sigma)$ in the sense of Definition~\ref{def:Kant}, 
satisfying
\begin{equation}\label{eq:Assphin}
|\varphi_n(x)|\leq C (1+\sfd^2(x, x_0)) \quad\forall x\in X
\end{equation}
and 
\begin{equation}\label{eq:Asspsin}
\psi_n\equiv -\infty \quad \text{on}\; X \setminus \supp \sigma\quad\text{and}\quad \psi_n (x)\leq C\,\,\forall x\in X. 
\end{equation}
Then there exist a subsequence $n(k)$ and a Kantorovich potential $\varphi=\psi^c$ of the transportation problem relative to 
$(\eta,\sigma)$ such that $\varphi_{n(k)} \to \varphi$ pointwise. In addition \eqref{eq:Assphin} is fulfilled by $\varphi$ and $\psi\leq C$.
\end{lemma}
\begin{proof} 
Since $X$ is separable, by the compactness properties of $\Gamma$-convergence 
we can assume with no loss of generality that $-\psi_n$ $\Gamma$-converges as $n\to\infty$, and we shall denote by $-\psi$ its
$\Gamma$-limit.  Observe that, since by definition of $\Gamma$-convergence for every $x \in X$ there exists a sequence $x_n\to x$ such that 
$-\psi_n(x_n)\to -\psi(x)$, $\psi$ still satisfies \eqref{eq:Asspsin}. 

By the invariance of $\Gamma$-convergence under continuous additive perturbations we get 
\begin{equation}
\left(\frac{1}{2} \sfd^2(x,\cdot)-\psi\right)= \Gamma-\lim_{n\to \infty} \left(\frac{1}{2} \sfd^2(x,\cdot)-\psi_n\right) \qquad \forall x \in X.
\end{equation} 
Because of \eqref{eq:Asspsin} and of the compactness of  $\supp\sigma$, we can use \eqref{eq:vuoleilreferee} to get
\begin{equation}\label{def:varphi}
\varphi_n(x)=\min_X \left(\frac{1}{2} \sfd^2(x,\cdot)- \psi_n\right) \to \min_X \left(\frac{1}{2}\sfd^2(x,\cdot)-\psi \right)=\varphi(x),
\end{equation}
where the last equality has to be understood as the definition of $\varphi(x)$. Obviously \eqref{eq:Assphin} is fulfilled
by $\varphi$, so that $\varphi\in L^1(X,f\mm)$. In connection with $\psi$, obviously its positive part is $\sigma$-integrable.

Now we claim that $\varphi=\psi^c$ is a Kantorovich potential for the limit transportation problem $(f\mm,\sigma)$; we have to prove that 
\begin{equation}\label{claim:phiKant}
\int_X \varphi\, \d (f\mm)+\int_X \psi \,\d\sigma\geq \frac 1 2 W^2_2 (f \mm, \sigma),
\end{equation} 
since this inequality provides at the same time also integrability of the negative part of $\psi$.
Since by assumption $\varphi_n=\psi_n^c$ is a Kantorovich potential for $(f_n \mm, \sigma)$, we already know that 
\begin{equation}\label{eq:phinKant}
\int_X \varphi_n\, \d (f_n\mm)+\int_X \psi_n \,\d\sigma = \frac{1}{2} W^2_2 (f_n\mm, \sigma).
\end{equation} 
Using (b) it is immediate to check the weak convergence of $f_n \mm$ to $f\mm$, so that 
(see for instance Proposition~2.5 in \cite{Ambrosio-Gigli11}) 
\begin{equation}\label{eq:LSCW22}
W^2_2(f\mm,\sigma) \leq \liminf_n W^2_2(f_n\mm,\sigma).
\end{equation}
Moreover, using (b) and \eqref{eq:Assphin}, the dominated convergence theorem gives
\begin{equation}\label{eq:convIntphin}
\int_X \varphi_n\, \d (f_n\mm) \to \int_X \varphi \, \d (f\mm).
\end{equation}
Finally, by the very definition of $\Gamma$-limit we have 
\begin{equation}\nonumber\label{eq:PsiGammaL}
-\psi(x)=\inf\left\{\liminf_{n\to\infty}-\psi_n(x_n)|\,x_n \to x \right\}\leq \liminf_{n\to\infty} -\psi_n(x).
\end{equation}
Moreover, by assumption \eqref{eq:Asspsin}, $-\psi_n\geq -C$. Hence Fatou's lemma gives
\begin{equation}\label{eq:convIntphicn}
\limsup_{n\to \infty} \int_X \psi_n \,\d\sigma\leq \int_X \psi\, \d \sigma.
\end{equation}
Putting together \eqref{eq:phinKant}, \eqref{eq:LSCW22}, \eqref{eq:convIntphin} and \eqref{eq:convIntphicn} we get  \eqref{claim:phiKant} as desired. 
\end{proof}

Let us close this section by discussing the geodesic structure of $(\probt{X},W_2)$, see
\cite[Theorem~2.10]{Ambrosio-Gigli11} or \cite{Lisini07}.
If $\mu_0, \,\mu_1\in\probt X$ are connected by a constant speed geodesic $\mu_t$ in $(\probt X, W_2)$,
then there exists $\ppi \in \prob{\geo(X)}$ with $(\e_t)_\sharp\ppi = \mu_t$ for all $t\in [0,1]$
and
\[
 W_2^2(\mu_s,\mu_t) = \int_{\geo(X)}\sfd^2(\gamma_s,\gamma_t)\,\d\ppi(\gamma)=
 (s-t)^2\int_{\geo(X)}\ell^2(\gamma)\,\d\ppi(\gamma)\qquad\forall s,\,t\in [0,1],
\]
where $\ell(\gamma)=\sfd(\gamma_0,\gamma_1)$ is the length of the geodesic $\gamma$.
The collection of all the measures $\ppi$ with the above properties is denoted by
$\gopt(\mu_0,\mu_1)$.
The measure $\ppi$ is not uniquely determined by $\mu_t$, unless $(X,\sfd)$ is non-branching. 
The relation between optimal geodesic plans and optimal Kantorovich plans is given by the fact that
$\gamma:=(\e_0,\e_1)_\sharp\ppi$ is optimal whenever $\ppi\in\gopt(\mu_0,\mu_1)$.

\subsection{Gradient flows}\label{ssevi}

In this section we review the notions of gradient flows in the metric sense, in the
$EVI_K$ sense and in the classical sense provided, in Hilbert spaces, by the theory of
monotone operators.

Let $(Y,\sfd_Y)$ be a complete and separable metric space and $K\in\R$. 
We say that $E:Y\to\R\cup\{+\infty\}$ is
$K$-geodesically convex if for any $y_0,\,y_1\in D(E)$ there exists
$\gamma\in\geo(Y)$ satisfying $\gamma_0=y_0$, $\gamma_1=y_1$ and
\[
E(\gamma_t)\leq (1-t)E(y_0)+tE(y_1)-\frac K2t(1-t)
\sfd_Y^2(y_0,y_1)\qquad\text{for every } t\in[0,1].
\]
%Notice that if $E$ is $K$-geodesically convex, then $D(E)$ is
%geodesic in $Y$ and therefore $\overline{D(E)}$ is a length space.
%
%A consequence of $K$-geodesic convexity is that the descending slope
%$\slopem  E$ can be calculated at all $y\in D(E)$ as
%\begin{equation}
%\label{eq:slopesup} \slopem E(y)=\sup_{z\in
%D(E)\setminus\{y\}}\left(\frac{E(y)-E(z)}{\sfd_Y (y,z)}+\frac K2
%\sfd_Y(y,z)\right)^+.
%\end{equation}
%We recall (see \cite[Corollary~2.4.10]{Ambrosio-Gigli-Savare08})
%that for $K$-geodesically convex and l.s.c.~functionals
%the descending slope is an upper gradient, in particular
%the property we shall need is
%\begin{equation}
%\label{eq:boundtuttecurve} E(y_s)\le E(y_t)+\int_s^t |\dot
%y_r|\,\slopem E(y_r)\,\d r \qquad\text{for every }s,\,t\in
%[0,\infty),\ s<t,
%%\end{equation}
%for \emph{all} locally absolutely continuous curves $y:[0,\infty)\to
%D(E)$. A metric gradient flow for the $K$-geodesically convex
%functional $E$ is a locally absolutely continuous curve
%$y:[0,\infty)\to D(E)$ along which \eqref{eq:boundtuttecurve} holds
%as an equality and moreover $|\dot y_t|=\slopem  E(y_t)$ for a.e.
%$t>0$, so that the energy dissipation rate $\tfrac{\d}{\dt}E(y_t)$
%is equal to $-|\dot y_t|^2=-\slopem  E^2(y_t)$ for a.e.~$t>0$.
%
%An application of Young inequality shows that metric gradient flows
%for $K$-geodesically convex and l.s.c. functionals can equivalently
%be defined as follows.
%
\begin{definition}[Metric  formulation of gradient flow]\label{def:dissKconv}
Let $E:Y\to\R\cup\{+\infty\}$ be a $K$-geodesically convex and
l.s.c. functional. We say that a locally absolutely continuous curve
$[0,\infty)\ni t\mapsto y_t\in D(E)$ is a gradient flow of $E$
starting from $y_0\in D(E)$ if
\begin{equation}\label{eq:ede}
E(y_0)= E(y_t)+\int_0^t\frac12 |\dot y_r|^2+\frac12|D^- E|^2(y_r)\,\d r\qquad\forall t\geq 0.
\end{equation}
\end{definition}
%

%A byproduct of this proof is also (see
%\cite[Theorem~9.3(i)]{Ambrosio-Gigli-Savare11}) the equality between
%slope and the so-called Fisher information functional:
%\begin{equation}\label{eq:slopeFisher}
%\slopem \entv^2(\rho\mm)=4\int\weakgrad{\sqrt{\rho}}^2\,\d\mm
%\end{equation}
%for all probability densities $\rho$ such that $\sqrt{\rho}\in
%D(\C)$. Choosing $f=\sqrt{\rho}$ this identity, in conjunction with
%the $HWI$ inequality relating entropy, Wasserstein distance and
%Fisher information (see \cite{Lott-Villani-Poincare} or
%\cite[Proposition~7.18]{Ambrosio-Gigli11}) gives the log-Sobolev
%inequality
%\begin{equation}\label{eq:logSobolev}
%\int f^2\log f^2\,\d\mm\leq\frac{2}{K}\int\weakgrad{f}^2\,\d\mm
%\qquad\text{whenever $f\in D(\C)$ and $\int f^2\,\d\mm=1$.}
%\end{equation}

Next we recall a stronger formulation of gradient flows, introduced and
extensively studied in \cite{Ambrosio-Gigli-Savare08},
\cite{Daneri-Savare08}.
\begin{definition}[Gradient flows in the $EVI_K$ sense]\label{def:EVIK}
Let $E:Y\to\R\cup\{+\infty\}$ be a lower semicontinuous functional,
$K\in\R$ and $(0,\infty)\ni t\mapsto y_t\in D(E)$ be a locally
absolutely continuous curve. We say that $(y_t)$ is a $K$-gradient
flow for $E$ in the Evolution Variational Inequalities sense (or,
simply, it is an $EVI_K$ gradient flow) if for any $z\in Y$ we have
\begin{equation}
\label{eq:defevi} \frac \d{\d t}\frac{\sfd_Y^2(y_t,z)}2+\frac
K2\sfd_Y^2(y_t,z)+E(y_t)\leq E(z)\qquad\text{for a.e.~$t\in
(0,\infty)$.}
\end{equation}
If $\lim\limits_{t\downarrow 0}y_t=y_0 \in \overline {D(E)}$, we say
that the gradient flow starts from $y_0$.
\end{definition}
Notice that the derivative in \eqref{eq:defevi} exists for
a.e.~$t>0$, since $t\mapsto\sfd_Y(y_t,z)$ is locally absolutely
continuous in $(0,\infty)$.

We recall some basic and useful properties of gradient flows in the
$EVI_K$ sense, see  Proposition~2.22 in \cite{Ambrosio-Gigli-Savare11b}; 
we also refer to \cite[Chap.\ 4]{Ambrosio-Gigli-Savare08} for more results. In particular, we
emphasize that the maps $\sfS_t:y_0\mapsto y_t$ that at every $y_0$
associate the value at time $t\ge0$ of the unique $K$-gradient flow
starting from $y_0$ give raise to a continuous semigroup of
$K$-contractions according to \eqref{eq:21} in a closed (possibly
empty) subset of $Y$.
\begin{proposition}[Properties of gradient flows in the $EVI_K$ sense]\label{prop:evipropr}
Let $Y$, $E$, $K$, $y_t$ be as in Definition~\ref{def:EVIK}
 and suppose that $(y_t)$ is an $EVI_K$ gradient flow of $E$ starting from
$y_0$. Then:
\begin{itemize}
\item[(i)] If $y_0\in D(E)$, then $y_t$ is also a metric gradient flow,
i.e.~\eqref{eq:ede} holds.
\item[(ii)] If $(\tilde y_t)$ is another $EVI_K$ gradient flow for $E$ starting
from $\tilde{y}_0$, then
\begin{equation}
\sfd_Y(y_t,\tilde y_t)\leq e^{-Kt}\sfd_Y(y_0,\tilde y_0).\label{eq:21}
\end{equation}
In particular, $EVI_K$ gradient flows uniquely depend on the
initial condition.
\item[(iii)]
Existence of $EVI_K$ gradient flows starting from any point in
$D\subset Y$ implies existence starting from any point in $\overline
D$.
\end{itemize}
\end{proposition}

If $(Y,\sfd_Y)$ is a Hilbert space with distance induced by the scalar product,
the gradient flow of a lower semicontinuous functional $E:Y\to \R\cup\{+\infty\}$
can also be defined as a locally absolutely continuous
map $y_t:(0,\infty)\to H$ satisfying
\begin{equation}\label{eq:defheat}
\frac{\d}{\d t}y_t\in -\partial^- E(y_t)\,\,\,\text{for a.e. $t>0$},
\qquad\lim_{t\downarrow 0}y_t=y\,\,\text{in $H$,}
\end{equation}
where the Frechet subdifferential $\partial^-E(y)$ is defined by
\begin{equation}\label{eq:subdiff}
\partial^- E(y):=\left\{\xi\in H:\ \liminf_{y'\to y}\frac{E(y')-E(y)-\langle\xi,y'-y\rangle}{\sfd_Y(y',y)}\geq 0\right\}.
\end{equation}
Under a $K$-convexity assumption the subdifferential can be equivalently defined
\begin{equation}\label{eq:subdiff1}
\partial^- E(y):=\left\{\xi\in H: E(y')\geq E(y)+\langle \xi,y'-y\rangle+\frac{K}{2}\sfd_Y^2(y',y)\,\,\,\text{for all $y'\in H$}\right\}.
\end{equation}
Differentiating the squared distance in \eqref{eq:defevi} yields that the $EVI_K$ formulation and \eqref{eq:defheat}
are equivalent in the Hilbert setting, for $K$-convex functionals.

\section{Weak gradients and weighted Cheeger energies}\label{sec:Cheeger}

In this section we recall the main results of the theory of weak gradients
as developed by the first two authors with Savar\'e in \cite{Ambrosio-Gigli-Savare11}, emphasizing the connections with the points of view 
developed by Cheeger in \cite{Cheeger00}, Koskela-MacManus in \cite{Koskela-MacManus} and 
Shanmugalingam in \cite{Shanmugalingam00}. We prove in Theorem~\ref{thm:change} the equivalence of weak gradients
defined with reference measures $\nn$ and $\mm$, under suitable assumptions on the density of $\nn$ w.r.t. $\mm$.
We introduce in \eqref{def:Cheeger} the weighted
Cheeger energy ${\C}_\nn$ and show in Theorem~\ref{thm:weighted} that, under the assumptions of Theorem~\ref{thm:change}, 
${\C}_\nn$ is quadratic whenever $\C$ is quadratic.

In the next two definitions we consider test plans and ``Sobolev" functions with respect to a reference nonnegative Borel measure
$\nn$ in $X$, finite on bounded sets. In the sequel we shall denote by $\Goodmeasures$ this class of measures, including both
probability measures and our reference measure $\mm$.

\begin{definition}[Test plan]
We say that $\ppi\in\Probabilities{C([0,1];X)}$ is a 2-test plan relative to $\nn\in\Goodmeasures$ if:
\begin{itemize}
\item[(i)] $\ppi$ is concentrated on $AC^2([0,1];X)$ and the $2$-action of $\ppi$ is finite:
$$
{\cal A}_2(\ppi):=\int\int_0^1|\dot\gamma_t|^2\,\d t\,\d\ppi(\gamma)<\infty.
$$ 
\item[(ii)] There exists $C\geq 0$ such that $(\e_t)_\sharp\ppi\leq C\nn$ for all $t\in [0,1]$.
\end{itemize}
\end{definition}

The following definition is inspired by the Heinonen-Koskela's concept \cite{Heinonen-Koskela98} of upper gradient, that we now illustrate.
A Borel function $G:X\to [0,\infty]$ is an upper gradient of a Borel function $f:X\to\R$ if 
$$
|f(\gamma_b)-f(\gamma_a)|\leq \int_a^bG(\gamma_s)|\dot\gamma_s|\,\d s
$$
for any absolutely continuous curve $\gamma:[a,b]\to X$. Since the inequality is invariant under reparameterization one can also
reduce to curves defined in $[0,1]$.

Let $\Curvesnonpara{X}$ be the set of continuous parametric curves $C\subset X$ with finite length, where curves
equivalent under reparameterization are identified. Recall that any such curve $C$ can be written
as $\gamma([0,\ell])$, where $\ell$ is the length of $C$ and $\gamma:[0,\ell]\to X$ is Lipschitz
with $|\dot\gamma|=1$ a.e. in $[0,\ell]$. We shall denote by
$i:AC^2([0,1];X)\to\Curvesnonpara{X}$ the natural surjection.

Recall also that the the $2$-modulus of $\Gamma\subset\Curvesnonpara{X}$ is defined by
\begin{equation}\label{eq:Mod2}
{\rm Mod}_{2,\nn}(\Gamma):=\inf\left\{\int_Xg^2\,\d\nn:\
\text{$g:X\to [0,\infty]$ Borel, $\int_\gamma g\geq 1$ for all $\gamma\in\Gamma$}\right\}.
\end{equation}
Shanmugalingam proved in \cite{Shanmugalingam00} that functions with an upper gradient in $L^2(X,\nn)$
are absolutely continuous along ${\rm Mod}_{2,\nn}$-a.e. curve in $\Curvesnonpara{X}$.
We also recall the following simple consequence of 
\eqref{eq:Mod2}: for any ${\rm Mod}_{2,\nn}$-negligible set $\Gamma$ there exist Borel functions $r_h:X\to [0,\infty]$
satisfying $\int_X r_h^2\,\d\nn\to 0$ and $\int_\gamma r_h=\infty$ for all $\gamma\in\Gamma$. Also, the inequality
$$
{\rm Mod}_{2,\nn}\bigl(\{\gamma:\ \int_\gamma g\geq t\}\bigr)\leq \frac{1}{t}\biggl(\int_Xg^2\,\d\nn\biggr)^{1/2}\qquad t>0
$$
immediately yields that functions in $L^2(X,\mm)$ have a finite integral on $\gamma$ for ${\rm Mod}_{2,\nn}$-a.e. $\gamma$.

\begin{definition}[The space $\calW_\nn$ and weak upper gradients]\label{def:wug}
Let $f:X\to\R$, $G:X\to [0,\infty]$ be Borel functions. We say that $G$ is a $2$-weak upper gradient relative to $\nn$ of $f$ if
$$
|f(\gamma_1)-f(\gamma_0)|\leq \int_0^1G(\gamma_s)|\dot\gamma_s|\,\d s<\infty\qquad\text{for $\ppi$-a.e. $\gamma$}
$$
for all $2$-test plans $\ppi$ relative to $\nn$. \\
We write $f\in\calW_\nn$ if $f$ has a $2$-weak upper gradient in $L^2(X,\nn)$. The $2$-weak upper gradient
relative to $\nn$ with minimal $L^2(X,\nn)$ norm (the so-called minimal $2$-weak upper gradient) will be denoted by $|D f|_{w,\nn}$.
\end{definition}

\begin{remark}[Sobolev regularity along curves]\label{rem:charaweakgrad}{\rm
A consequence of $\calW_\nn$ regularity is (see Proposition~5.7 in \cite{Ambrosio-Gigli-Savare11}) the Sobolev property along curves, namely for any
$2$-test plan $\ppi$ relative to $\nn$ the function $t\mapsto f(\gamma_t)$ belongs to the Sobolev space $W^{1,1}(0,1)$ and
$$
|\frac{\d}{\d t}f(\gamma_t)|\leq\weakgrad{f}(\gamma_t)|\dot\gamma_t|\qquad\text{a.e. in $(0,1)$}
$$ 
for $\ppi$-a.e. $\gamma$. Conversely, assume that $g$ is Borel nonnegative, that for any $2$-test plan $\ppi$ the map 
$t\mapsto f(\gamma_t)$ is $W^{1,1}(0,1)$ and that
$$
|\frac{\d}{\d t}f(\gamma_t)|\leq g(\gamma_t)|\dot\gamma_t|\qquad\text{a.e. in $(0,1)$}
$$ 
for $\ppi$-a.e. $\gamma$. Then, the fundamental theorem of calculus in $W^{1,1}(0,1)$ gives that $g$ is a $2$-weak upper gradient of $f$.
}\fr\end{remark}

Because of the absolute continuity condition $(\rme_t)_\sharp\ppi\ll\nn$ imposed on test plans, it is immediate to check that
the property of being in $\calW_\nn$, as well as $|D f|_{w,\nn}$, are invariant under modifications of $f$ in $\nn$-negligible sets.
Furthermore, these concepts are easily seen to be local with respect to $\nn$ in the following sense: if $f\in\calW_\nn$ then 
$f\in\calW_{\nn'}$ for all measures $\nn'=\nn\res B$ with $B\subset X$ Borel, and $|Df|_{w,\nn'}\leq |Df|_{w,\nn}$ $\nn'$-a.e. on $B$: this is due to the
fact that test plans relative to $\nn'$ are test plans relative to $\nn$. Conversely,
\begin{equation}\label{eq:localitynn}
\text{$f\in\calW_{\nn_R}$ with $\nn_R:=\nn\res\overline{B}_R(x_0)$, $\sup_R\int_X|D f|_{w,\nn_R}^2\,\d\nn_R<\infty$}
\quad\Longrightarrow\quad f\in\calW_\nn.
\end{equation}
This is due to the fact that any curve is bounded, hence any test plan $\ppi$ relative to $\nn$ can be monotonically approximated
by test plans concentrated on curves contained in a bounded set.

Another property we shall need is the locality with respect to $f$, see \cite{AGSBaEm} for the simple proof.

\begin{proposition}[Locality]\label{prop:locality}
Let $f_1,\,f_2:X\to \R$ Borel and let $G_1,\,G_2\in L^2(X,\nn)$ be $2$-weak upper gradients of $f_1,\,f_2$ relative to 
$\nn$ respectively.
Then 
$$\tilde{G}_1:= \begin{cases}
G_1&\text{on $\{f_1\neq f_2\}$;}\\
\min\{G_1,G_2\}&\text{on $\{f_1=f_2\}$}
\end{cases}
$$
is a $2$-weak upper gradient of $f_1$. In particular, by minimality we get
\begin{equation}\label{eq:locality}
|D f_1|_{w,\nn}=|D f_2|_{w,\nn}\qquad\text{$\nn$-a.e. on $\{f_1=f_2\}$.}
\end{equation}
\end{proposition}

Weak gradients share with classical gradients many features, in particular the chain rule
\cite[Proposition~5.14]{Ambrosio-Gigli-Savare11}
\begin{equation}\label{eq:chainrule}
|D \phi(f)|_{w,\nn}=\phi'(f)|D f|_{w,\nn}\qquad\text{$\nn$-a.e. in $X$}
\end{equation}
for all $\phi:\R\to\R$ Lipschitz and nondecreasing on an interval containing the image of $f$. 
By convention, as in the classical chain rule, $\phi'(f)$ is arbitrarily defined at all points
$x$ such that $\phi$ is not differentiable at $x$, taking into account the fact that
$|D f|_{w,\nn}=0$ $\nn$-a.e. on this set of points. 

In the sequel we shall adopt the conventions
\begin{equation}\label{eq:convention}
\weakgrad{f}:=|D f|_{w,\mm},\qquad\qquad\calW:=\calW_\mm.
\end{equation}
In Theorem~\ref{thm:equivalence} below we analyze in detail, the behaviour of $|D f|_{w,\nn}$ and $\calW_\nn$ 
under modifications of the reference measure $\nn$. 

\begin{theorem}\label{thm:equivalence}
The following properties hold:
\begin{itemize}
\item[(a)] If $\nn\in\Goodmeasures$ and $\Gamma\subset\Curvesnonpara{X}$ is ${\rm Mod}_{2,\nn}$-negligible, 
then any Borel set $\tilde\Gamma\subset AC^2([0,1];X)$ such that $i(\tilde\Gamma)\subset\Gamma$
is $\ppi$-negligible for any $2$-test plan $\ppi$ relative to $\nn$. In addition, for any Borel and $\nn$-negligible set
$N\subset X$ the following holds:
$$
{\rm Mod}_{2,\snn}\bigl(\bigl\{\gamma\in\Curvesnonpara{X}:\ \int_{\gamma^{-1}(N)}|\dot\gamma|\,\d t>0\bigr\}\bigr)=0.
$$
\item[(b)] If either $\nn\in\Probabilities{X}$ and $f\in\calW_\nn$, or $\nn\in\Goodmeasures$ and $f\in\calW_\nn\cap L^1(X,\nn)$, 
there exist $\phi_n\in {\rm Lip}_b(X)\cap L^2(X,\nn)$ satisfying $\phi_n\to f$ $\nn$-a.e. in $X$ and $|D\phi_n|\to |D f|_{w,\nn}$ in $L^2(X,\nn)$. 
\item[(c)] If either $\nn\in\Probabilities{X}$ and $f\in\calW_\nn$, or $\nn\in\Goodmeasures$ and $f\in\calW_\nn\cap L^1(X,\nn)$,
then there exists a Borel function $\tilde f$ coinciding with $f$ out of an $\nn$-negligible set and having an upper gradient in $L^2(X,\nn)$; 
in addition, there exist upper gradients $G_n$ of $\tilde{f}$ converging to $|D f|_{w,\nn}$ in $L^2(X,\nn)$.
\end{itemize}
\end{theorem} 
\begin{proof}
(a) The first statement is a simple consequence of H\"older inequality, see \cite[Remark~5.3]{Ambrosio-Gigli-Savare11}. The second one
follows just by taking the function $g$ identically equal to $\infty$ on $N$ and null out of $N$ in \eqref{eq:Mod2}.

(b) Using the chain rule \eqref{eq:chainrule} we reduce the proof to the case of nonnegative functions $f$. If 
$f$ belong to $L^2(X,\nn)$ the existence of $\phi_n$ is one of the main results of \cite{Ambrosio-Gigli-Savare11}, see Theorem~6.2 therein. 
In the general case we approximate $f$ by the truncated functions $f_N=\min\{f,N\}$
and use the chain rule again to show $|D f_N|_{w,\nn}\to |D f|_{w,\nn}$ in $L^2(X,\nn)$. 
Then, a diagonal argument provides the result.

(c) This is part of the theory developed by Koskela-MacManus in \cite{Koskela-MacManus} and 
Shanmugalingam in \cite{Shanmugalingam00}: if $f_n\to f$ $\nn$-a.e. and $G_n$ are upper gradients of
$f_n$ weakly convergent to $G$ in $L^2(X,\nn)$, then we can find a Borel function $\tilde{f}$ equal to $f$ $\nn$-a.e.
and a Borel function $\tilde G$ equal to $G$ $\nn$-a.e. such that  $\tilde{G}$ satisfies the upper gradient property relative to
$\tilde{f}$ along  ${\rm Mod}_{2,\nn}$-almost every curve. In our case when $f\in\calW_\nn$ we may apply statement (b) 
with $G=|D f|_{w,\nn}$ and choose $f_n=\phi_n$ to find $\tilde{f}$ and $\tilde{G}$. Then, denoting by $\Gamma$ the set of curves where
the upper gradient property fails and considering
$$
G_h:=\tilde{G}+r_h,
$$
where $r_h\in L^2(X,\nn)$ satisfy $\int_Xr_h^2\,\d\nn\to 0$ and $\int_\gamma r_\epsilon=\infty$ for all $\gamma\in\Gamma$,
we obtain upper gradients $G_h$ of $\tilde{f}$ approximating $|D f|_{w,\nn}$ in $L^2(X,\nn)$. 
\end{proof}

\begin{theorem}[Change of reference measure]\label{thm:change}
Assume that $\rho=g\mm\in\probt{X}$ with $g\in L^\infty(X,\mm)$ and $\weakgrad{\sqrt{g}}\in L^2(X,\mm)$.
Then:
\begin{itemize}
\item[(a)] $f\in\calW$ and $\weakgrad{f}\in
L^2(X,\rho)$ imply $f\in\calW_\rho$ and $|D
f|_{w,\rho}=\weakgrad{f}$ $\rho$-a.e. in $X$;
\item[(b)] $\log g\in\calW_\rho$ and $|D\log g|_{w,\rho}=\weakgrad g/g$ $\rho$-a.e. in $X$.    
\end{itemize}
\end{theorem}
\begin{proof}  (a) Thanks to the locality properties with respect to $\mm$ stated after Definition~\ref{def:wug} (see in particular
\eqref{eq:localitynn}) we can reduce ourselves to the case when $\mm(X)=1$. 
Since the statement is invariant under
modification of $f$ and $g$ in $\mm$-negligible sets, 
by Theorem~\ref{thm:equivalence}(b) we can assume that $\sqrt{g}$ and $f$ are absolutely continuous along
${\rm Mod}_{2,\mm}$-almost every curve in $\Curvesnonpara{X}$; even
more, we can assume that $f$ has an upper gradient
$H$ with $\int H^2\,\d\mm<\infty$.

Let us prove first the inequality $|D f|_{w,\rho}\leq\weakgrad{f}$ $\rho$-a.e. in $X$. By a truncation argument
we can assume with no loss of generality that $f$ is bounded; under this assumption we can find
bounded Lipschitz functions $\phi_n$ with $|D\phi_n|\to \weakgrad{f}$ in $L^2(X,\mm)$. Since $g$ is bounded it
follows that $|D\phi_n|\to\weakgrad{f}$ in $L^2(X,\rho)$; we can now use the stability properties of weak upper
gradients \cite[Theorem~5.12]{Ambrosio-Gigli-Savare11} to obtain that  $|D f|_{w,\rho}\leq\weakgrad{f}$ $\rho$-a.e. in $X$.

In order to prove the converse inequality $|D f|_{w,\rho}\geq\weakgrad{f}$ $\rho$-a.e. in $X$, 
we consider a function $\tilde{f}$ coinciding with $f$ $\rho$-a.e. in $X$ and an upper gradient $L$ of $\tilde{f}$ with 
$\int L^2\,\d\rho<\infty$. The converse inequality
follows by letting $L\to |Df|_{w,\rho}$ in $L^2(X,\rho)$, if we are able to show that
$$
L_1(x):=
\begin{cases}
H(x) &\text{if $g(x)=0$;}\\
\min\{H(x),L(x)\} &\text{if $g(x)>0$,}
\end{cases}
$$ 
is a $2$-weak upper gradient of $f$ relative to $\mm$. More precisely, we will prove that the upper gradient inequality with $L_1$
in the right hand side holds along ${\rm Mod}_{2,\mm}$-almost every curve. We notice first that 
$$
|\tilde{f}(\gamma_{\ell(\gamma)})-\tilde{f}(\gamma_0)|\leq\int_\gamma L
$$
along ${\rm Mod}_{2,\mm}$-a.e. curve $\gamma$ satisfying $\inf_\gamma g>0$ (here we are using
the invariance under reparameterization, selecting the
arclength one, with $\ell(\gamma)$ equal to the length of $\gamma$). Indeed, by definition of 2-modulus, the set
$$
\left\{\gamma\in\Curvesnonpara{X}:\ \inf_\gamma g>0,\,\,\int_\gamma L=\infty\right\}
$$
is not only ${\rm Mod}_{2,\rho}$-negligible, but also ${\rm Mod}_{2,\mm}$-negligible. If we write
the upper gradient inequality in averaged form 
$$
\frac{1}{\epsilon\ell(\gamma)}\int_0^{\epsilon\ell(\gamma)}
|\tilde{f}(\gamma_{\ell(\gamma)-r})-\tilde{f}(\gamma_r)|\,\d r\leq\int_\gamma L
\quad\text{with}\quad\epsilon<\frac 12
$$ 
and use Theorem~\ref{thm:equivalence}(a) with the $\mm$-negligible set $N=\{f\neq\tilde{f}\}\cap\{g>0\}$,
we may replace $\tilde{f}$ with $f$ in the previous inequality. Now we use the absolute continuity of $f$ along
${\rm Mod}_{2,\mm}$-a.e. curve and pass to the limit along a sequence $\epsilon_k\downarrow 0$ to get
$$
|f(\gamma_b)-f(\gamma_a)|\leq\int_\gamma L
$$
along ${\rm Mod}_{2,\mm}$-a.e. curve $\gamma:[a,b]\to X$ with $\inf_\gamma g>0$. 

The set of curves $\gamma\in\Curvesnonpara{X}$ containing a subcurve $\gamma':[a,b]\to X$ with $\inf_{\gamma'}g>0$ 
and $|f(\gamma_b')-f(\gamma_a')|>\int_{\gamma'} L$ is ${\rm Mod}_{2,\mm}$-negligible as well. If $\gamma$ does not
belong to this set and $f\circ\gamma$ is absolutely continuous, it is immediate to check 
(recall that $g$ is continuous along ${\rm Mod}_{2,\mm}$-almost every curve) that its derivative
is bounded a.e. by $L_1\circ\gamma|\dot\gamma|$, whence the upper gradient inequality along $\gamma$ follows.

(b) We consider the functions $f_\eps=\log (g+\eps)$. Since $\weakgrad{g}^2/g^2\in L^1(X,\rho)$ it
is immediate to check that all functions $f_\eps$ satisfy the assumption in (a), hence $f_\eps\in\calW_\rho$ and
$|D f_\eps|_{w,\rho}=|D f_\eps|_w=\weakgrad{g}/(g+\eps)$ $\rho$-a.e. in $X$. We can now pass to the limit
as $\eps\downarrow 0$ and use again the stability of weak upper gradients to get 
$|D f|_{w,\rho}\leq\weakgrad{g}/g$ $\rho$-a.e. in $X$. The converse inequality follows by the chain rule \eqref{eq:chainrule}
with $\phi(s):=\log(\rme^s+1)$:
$$
\frac{\weakgrad{g}}{g+1}=|Df_1|_{w,\rho}=\phi'(f)|Df|_{w,\rho}=\frac{g}{g+1}|Df|_{w,\rho}.
$$
\end{proof} 

\begin{remark}{\rm Notice that for the validity of (a) it suffices, as the proof shows, the existence
of a nonnegative function $\tilde{g}$ continuous along ${\rm Mod}_{2,\mm}$-a.e. curve and
satisfying $\mm(\{g\neq\tilde g\})=0$. 
}\fr\end{remark}

We shall define $\C:L^1(X,\mm)\to [0,\infty]$, ${\C}_\nn:L^1(X,\nn)\to [0,\infty]$ by
\begin{equation}\label{def:Cheeger}
\C(f):=\frac{1}{2}\int_X \weakgrad{f}^2\,\d\mm\quad f\in\calW,\qquad
{\C}_\nn(f):=\frac{1}{2}\int_X |D f|_{w,\nn}^2\,\d\nn\quad f\in\calW_\nn
\end{equation}
with the conventions $\C(f)=\infty$ on $L^1(X,\mm)\setminus\calW$, ${\C}_\nn(f)=\infty$ on
$L^1(X,\nn)\setminus\calW_\nn$. We will choose $\nn$, as explained in the introduction, to be
probability measures. 

We shall also denote, whenever $\C$ (resp. $\C_\nn$) is a quadratic form, by
\begin{equation}\label{eq:numeriamoanchequesta}
\mathcal E(f,g):=\frac{1}{2}\bigl(\C(f+g)-\C(f-g)\bigr)\qquad
\biggl(\text{resp. } \mathcal E_\nn(f,g):=\frac{1}{2}\bigl({\C}_\nn(f+g)-{\C}_\nn(f-g)\bigr)\biggr)
\end{equation}
the associated symmetric bilinear form, defined on $\calW\cap L^1(X,\mm)$ (resp. $\calW_\nn\cap L^1(X,\nn)$).

Still under the assumption that $\C$ is quadratic, as in \cite[Definition~4.13]{Ambrosio-Gigli-Savare11b} 
(see also Gigli's work \cite{Gigli12} for a more general, non-quadratic framework) we can define
\begin{equation}\label{eq:defGamma}
\Gbil{f}{g}:=\lim_{\eps\downarrow 0}\frac{\weakgrad{(f+\eps g)}^2-\weakgrad{f}^2}{2\eps}
\qquad f,\,g\in\calW,
\end{equation}
where the limit takes place in $L^1(X,\mm)$. Notice that $\Gbil{f}{f}=\weakgrad{f}^2$ $\mm$-a.e. and that
$\Gbil{\cdot}{\cdot}$ provides integral representation
to $\mathcal E$, namely 
$$
\mathcal E(f,g)=\int_X\Gbil{f}{g}\,\d\mm.
$$
The inequality $\weakgrad{(f+\eps g)}^2\leq \bigl(\weakgrad{f}+\eps\weakgrad{g}\bigr)^2=
\weakgrad{f}^2+2\eps\weakgrad{f}\weakgrad{g}+\eps^2\weakgrad{g}^2$ provides the bound
\begin{equation}\label{eq:boundGamma}
\big|\Gbil{f}{g}\bigr|\le\weakgrad{f}\weakgrad{g}\qquad\text{$\mm$-a.e. in $X$.}
\end{equation}
Also, locality of weak gradients gives
\begin{equation}\label{eq:localityGamma}
\Gbil{f}{g}=\Gbil{f}{g'}\qquad\text{$\mm$-a.e. on $\{g=g'\}$.}
\end{equation}
%It could be proved that $\Gbil{\cdot}{\cdot}$ is a symmetric bilinear form, but the proof (see \cite{Ambrosio-Gigli-Savare11b}
%for finite measures and \cite{Gigli12} for the general case) is not elementary and we shall not need this fact in the paper.
We will need a chain rule with respect to the
second argument, see \cite[Lemma~4.7]{Ambrosio-Gigli-Savare11b} for the simple proof:
\begin{equation}\label{eq:chainriemannian}
\int_X\Gbil{f}{\phi(g)}\,\d\mm=\int_X\phi'(g)\Gbil{f}{g}\,\,d\mm%\qquad\text{$\mm$-a.e. in $X$}
\end{equation}
for all $\phi:\R\to\R$ nondecreasing and Lipschitz on an interval containing the image of $g$, 
with the same convention on the value of $\phi'(g)$ mentioned in \eqref{eq:chainrule}. 
Finally, we will need the following lemma, whose proof is more delicate: it relies on the chain rule for
$\Gbil{\cdot}{\cdot}$ also with respect to the first factor and on the Leibniz rule with respect to the second factor 
(see \cite{Ambrosio-Gigli-Savare11b} for finite measures and \cite[Proposition~4.20]{Gigli12} for the general case).

\begin{lemma}\label{lem:vabbeserve}
If $\C$ is quadratic, then $\Gbil{\cdot}{\cdot}$ is a symmetric bilinear form. In particular
$\int\weakgrad{f}^2 g\,\d\mm=\int\Gbil{f}{f}g\,\d\mm$ is a quadratic form for any nonnegative $g\in L^\infty(X,\mm)$.
\end{lemma} 

\begin{theorem}[Weighted Cheeger energy]\label{thm:weighted}
Assume that $\rho=g\mm\in\probt{X}$ with $g\in L^\infty(X,\mm)$ and $\C(\sqrt{g})<\infty$. If $\C$ is a quadratic
form, then $\C_\rho$ is a quadratic form and 
\begin{equation}\label{eq:transfer1}
\mathcal E_\rho(\log g,\varphi)=\mathcal E(g,\varphi)\qquad\text{
for all $\varphi:X\to\R$ Lipschitz with bounded support.}
\end{equation}
\end{theorem}
\begin{proof} By Theorem~\ref{thm:change}(a) and Lemma~\ref{lem:vabbeserve}, 
$\C_\rho$ is a quadratic form on bounded Lipschitz functions with bounded 
support. By approximation $\C_\rho$ is a quadratic form
on bounded Lipschitz functions and eventually, taking 
Theorem~\ref{thm:equivalence}(b) into account, on $L^2(X,\rho)$. \\
Let $f_\eps=\log(g+\eps)\in\calW$. Then, using again the independence of weak gradients upon the reference measure given
by Theorem~\ref{thm:change}(a) and \eqref{eq:chainriemannian}, we get
\begin{eqnarray*}
\mathcal E_\rho(\varphi,f_\eps)&=&\lim_{\delta\downarrow 0} \frac{{\C}_\rho(\varphi+\delta f_\eps)-{\C}_\rho(\varphi)}{\delta}=
\lim_{\delta\downarrow 0}\int_X\frac{\weakgrad{(\varphi+\delta f_\eps)}^2-\weakgrad{\varphi}^2}{2\delta}\,\d\rho\\&=&
\int_X\Gbil{\varphi}{f_\eps}\,\d\rho=\int_X\Gbil{\varphi}{g}\frac{g}{g+\eps}\,\d\mm.
\end{eqnarray*}
Passing to the limit as $\eps\downarrow 0$ provides the result, since convergence of the right hand sides is obvious, while
convergence of the left hand sides can be obtained working in the vector space
$H:=L^2(X,\rho')\cap\calW_\rho$ endowed with the scalar product
$$
\langle h,h'\rangle:=\int_X hh'\,\d\rho'+\mathcal E_\rho(h,h')
\quad\text{with}\quad\rho':=\frac{1}{1+\log^2 g}\rho.
$$
This is indeed a Hilbert space because $\C_\rho$ is easily seen to be lower semicontinuous (since a truncation
argument allows the reduction to sequences uniformly bounded in $L^\infty(X,\rho)$) also w.r.t. $L^2(X,\rho')$ 
convergence; moreover, clearly $f_\eps\to f$ in $L^2(X,\tilde\rho)$ and  since their norms are uniformly bounded we have
weak convergence in $H$. Finally $g\mapsto\mathcal E_\rho(\varphi,g)$ is continuous in $H$.
\end{proof}

\section{Existence of good geodesics}\label{sec:Tapio}

This section is devoted to the proof of the existence of geodesics
in $(\probt{X},W_2)$ which are (at least for some initial time 
interval) better than the ones given directly by the usual 
$CD(K,\infty)$ inequality given by Lott and Villani \cite{Lott-Villani09} and Sturm \cite{Sturm06I}.

\begin{definition}\label{def:CD}
We say that $(X,\sfd,\mm)$ is a $CD(K,\infty)$ space if, 
for all $\mu_0,\,\mu_1 \in D(\entv)$ (recall \eqref{def:DDDD}) there exists 
a geodesic $(\mu_t) \in \geo(\ProbabilitiesTwo X)$ which 
satisfies the convexity inequality
\begin{equation}\label{eq:CDdef}
 \entv(\mu_t) \le (1-t)\entv(\mu_0) + t \entv(\mu_1)
                    - \frac{K}{2}t(1-t)W_2^2(\mu_0,\mu_1)\qquad\forall t\in [0,1].
\end{equation}
\end{definition}

The idea of constructing good geodesics in $CD(K,N)$ spaces was 
recently used by Rajala in \cite{R2011b} to study $CD(K,N)$ spaces with branching
geodesics. There the initial motivation was to 
obtain geodesics good enough so that the approach of \cite{R2011} 
for proving local Poincar\'e inequalities could be adapted to these 
spaces. Constructing geodesics by selecting midpoints is a standard approach,
 see for example Gromov's proof that the GH limit of length spaces
is a length space \cite[Proposition 3.8]{Gromov07}.% and more recently the work of Bacher and Sturm in \cite{BS2010}.

Here we modify some of Rajala's results \cite{R2011b} and \cite{R2012} to 
the setting of this paper, repeating with some details the arguments 
because on some occasions the adaptation is not trivial. 
The version of these results which we will need in the later 
sections is the following.

\begin{theorem}\label{thm:goodgeodesics}
 Let $(X,\sfd,\mm)$ be a $CD(K,\infty)$ space and let 
 $\mu_0=\rho_0\mm,\,\mu_1=\rho_1\mm \in D(\entv)$. 
 Assume in addition that $\mu_1$ has
 bounded support and density and that the density $\rho_0$ satisfies the growth-bound
 \begin{equation}\label{eq:decay}
  \rho_0(x) \le c_1\rme^{-c_2\sfd^2(x,x_0)}\qquad\forall x\in X
 \end{equation}
 for some $c_1,\,c_2>0$ and $x_0 \in X$.

 Then there exist $t_0 \in (0,1)$ and a geodesic
 $(\mu_t) \in \geo(\ProbabilitiesTwo X)$ between 
 $\mu_0$, $\mu_1$ satisfying the convexity inequality
 \eqref{eq:CDdef} for all $t \in [0,1]$ and the density bound
 \begin{equation}\label{eq:uniformbound}
  \sup_{t \in [0,t_0]}||\rho_t||_{L^\infty(X,\mm)} < \infty.
 \end{equation}
\end{theorem}

%We will construct the geodesic of Theorem~\ref{thm:goodgeodesics} by
%connecting measures of minimal entropy. A similar construction was done 
%recently by Rajala in \cite{R2012} in $CD^*(K,N)$ spaces. 

In $\S$\ref{ss41} we discuss the convexity of the entropy along intermediate measures formed using
an inductive process and prove existence of entropy minimizers. In $\S$\ref{ss42}
we review some result of Rajala  in \cite{R2012} in $CD^*(K,N)$ spaces. In $\S$\ref{ss43}
we prove that the minimizers satisfy density bounds by adapting Rajala's result in \cite{R2011b}. Finally,
in $\S$\ref{ss44} we prove Theorem~\ref{thm:goodgeodesics} using these ingredients.

\subsection{Intermediate measures and the existence of minimizers}\label{ss41}

The measures with minimal entropy will be selected from the set of all 
intermediate measures. Recall that for any two measures 
$\mu_0,\, \mu_1 \in \ProbabilitiesTwo X$ the set of all intermediate 
points (with a parameter $t \in (0,1)$), will be denoted by
\begin{align*}
 \mathcal{I}_t(\mu_0,\mu_1) = \{\nu \in \ProbabilitiesTwo X
   \,:\, W_2(\mu_0,\nu) = t W_2(\mu_0,\mu_1) \text{ and }
   W_2(\mu_1,\nu) = (1-t) W_2(\mu_0,\mu_1)\}. 
\end{align*}
It is not difficult to show that the set of $t$-intermediate points is a convex and closed subset
of $\ProbabilitiesTwo X$,

Even though the selection process is countable, it will define the whole 
geodesic by completion. To get the convexity inequality \eqref{eq:CDdef} 
for all times we will then need the lower semicontinuity of the entropy
w.r.t. $W_2$-convergence (a direct consequence of \eqref{eq:changeentropy} and of
the weak lower semicontinuity of ${\rm Ent}_\nn$ in $\prob{X}$ when $\nn\in\Probabilities{X}$)
and tightness estimates.
Let us now indicate how the first property of the good geodesics follows 
easily if we define the geodesic by taking any intermediate point where 
\eqref{eq:CDdef} is satisfied.

\begin{proposition}\label{prop:combinedgeod}
 Let $\mu_0, \,\mu_1 \in \ProbabilitiesTwo X$. Suppose that we have 
 selected inductively at step $(n+1)$ measures
 $\mu_{t} \in \mathcal{I}_{\frac{t-s}{r-s}}(\mu_s,\mu_r)$ satisfying
 \[
   \entv(\mu_{t}) \le \frac{(r-t)}{(r-s)}\entv(\mu_s)
    + \frac{(t-s)}{(r-s)} \entv(\mu_r)
    - \frac{K}{2}\frac{(t-s)}{(r-s)}\frac{(r-t)}{(r-s)}W_2^2(\mu_s,\mu_r),
 \]
 where $s < t < r$ and the times $s$ and $r$ are two consecutive
 timepoints in the set of times where the measures have already been
 selected at step $n$.
 
 Then \eqref{eq:CDdef} holds for all $\mu_t$ chosen at the $(n+1)$-th step. 
 In particular, if the closure of the selected times is the whole interval $[0,1]$, defining
 $\mu_t$ by completion, we have a geodesic between $\mu_0$ and $\mu_1$ along which
 \eqref{eq:CDdef} holds.
\end{proposition}

\begin{proof} Suppose that we have selected a measure
 $\mu_{t} \in \mathcal{I}_{t}(\mu_0,\mu_1)$ satisfying 
 \[
  \entv(\mu_t) \le (1-t)\entv(\mu_0)
   + t \entv(\mu_1) - \frac{K}{2}t(1-t)W_2^2(\mu_0,\mu_1)
 \]
 and after it a measure $\mu_{ts} \in \mathcal{I}_{s}(\mu_0,\mu_t)$
 satisfying 
 \[
  \entv(\mu_{ts}) \le (1-s)\entv(\mu_0)
   + s \entv(\mu_t) - \frac{K}{2}s(1-s)W_2^2(\mu_0,\mu_t).
 \]
 Then for the measure $\mu_{ts}$ we also have
 $\mu_{ts} \in \mathcal{I}_{ts}(\mu_0,\mu_1)$ and
 \begin{align*}
  \entv&(\mu_{ts}) \le (1-s)\entv(\mu_0) + s \entv(\mu_t) - \frac{K}{2}s(1-s)W_2^2(\mu_0,\mu_t) \\
                   \le\,& (1-s)\entv(\mu_0) + s \left((1-t)\entv(\mu_0) + t \entv(\mu_1) - \frac{K}{2}t(1-t)W_2^2(\mu_0,\mu_1) \right)\\
                      & - \frac{K}{2}s(1-s)W_2^2(\mu_0,\mu_t) \\
                   = \,& \left((1-s)+s(1-t)\right)\entv(\mu_0) + ts\entv(\mu_1) - \frac{K}{2}\left(ts(1-t)+t^2s(1-s)\right)W_2^2(\mu_0,\mu_1)\\
                   = \,& (1-ts)\entv(\mu_0) + ts\entv(\mu_1) - \frac{K}{2}ts(1-ts)W_2^2(\mu_0,\mu_1).
 \end{align*}
 Therefore the claim holds for all the points $t_i$. By the lower semicontinuity of the entropy it then holds also for the closure.
\end{proof}

Now that we know from Proposition~\ref{prop:combinedgeod} that
the first property of the geodesic in Theorem~\ref{thm:goodgeodesics}
is easily satisfied we turn to the more difficult part of obtaining
the density bound \eqref{eq:uniformbound}. To do this we will not 
only select intermediate measures that satisfy \eqref{eq:CDdef}, 
but measures where the entropy is minimal. The obvious first step is 
then to prove that there indeed exist such minimizers. In general the set 
$\mathcal{I}_t(\mu_0,\mu_1)$, though closed, is not compact in 
$(\ProbabilitiesTwo X, W_2)$. However, when we consider a subset of 
$\mathcal{I}_t(\mu_0,\mu_1)$ with the entropy bounded from above, we 
have compactness. In particular, we therefore have the existence of 
minimizers.

\begin{lemma}\label{lma:minexists}
 Let $\mu_0, \,\mu_1 \in \probt{X}$. 
 Then for all $t\in[0,1]$ there exists a minimizer of the entropy in $\mathcal{I}_t(\mu_0,\mu_1)$.
\end{lemma}
\begin{proof}
 Without loss of generality we can assume the existence of 
 $\nu \in \mathcal{I}_t(\mu_0,\mu_1)$ with $\entv(\nu)<\infty$.
 We know that the entropy is lower 
 semicontinuous and that
 $\mathcal{I}_t(\mu_0,\mu_1)$ is closed. The claim then follows
 if we are able to show that the set
 \[
  \mathcal{K} = \{\mu \in \mathcal{I}_t(\mu_0,\mu_1)\,:\,
  \entv(\mu) \le \entv(\nu)\} \subset \ProbabilitiesTwo X
 \]
 is relatively compact in $(\ProbabilitiesTwo X, W_2)$. It suffices
 to prove that the set $\mathcal{K}$ is uniformly $2$-integrable
 and tight, see \cite[Proposition 7.15]{Ambrosio-Gigli-Savare08}.
 Let us first prove the uniform $2$-integrability of the set 
 $\mathcal{I}_t(\mu_0,\mu_1)$. This follows from the fact that for
 any $\mu \in \mathcal{I}_t(\mu_0,\mu_1)$ we have
 \[
  \int_{X \setminus \overline{B}(x_0,k)} \sfd^2(x_0,x)\,\d\mu
         \le \int_{X \setminus \overline{B}(x_0,k/2)} 4\sfd^2(x_0,x)\,\d(\mu_0+\mu_1)
         \to 0, \quad \text{as }k \to \infty
 \]
 since $\mu_0,\,\mu_1 \in \ProbabilitiesTwo X$.
 
 Let us next prove that $\mathcal{K}$ is tight. If $\tilde\mm\in\Probabilities{X}$ is defined as in \eqref{eq:grygorian2},
 \eqref{eq:changeentropy} shows that $\sup_{\mu\in\mathcal K}{\rm Ent}_{\tilde\mm}(\mu)$ is finite. Then,
 tightness of $\mathcal K$ is a simple consequence of the equi-integrability of the densities w.r.t. $\tilde\mm$. 
 %
 %\[
  %\hat\mm = e^{-c_1\sfd^2(x_0,x)+c_2}\mm.
 %\]
 %From \eqref{eq:stras} we know that $\hat\mm$ is a probability
 %measure. For any $\mu =  \rho\mm \in \mathcal{K}$ we can write
 %\[
 % \rho\mm = \hat\rho\hat\mm =\hat\rho e^{-c_1\sfd^2(x_0,x)+c_2}\mm
 %\]
 %and notice that
 %\[
 % \entv(\mu) = \int_X\rho \log \rho \,\d\mm
  %           = \int_X\hat\rho \log\hat\rho\,\d\hat\mm
  %             - \int_Xc_1\Wgh^2(x)\rho\,\d\mm + c_2,
 %\]
 %where
 %\[
  %\int_Xc_1\Wgh^2(x)\rho\,\d\mm < C
 %\]
 %with the bound $C$ independent of the choice of
 %$\mu \in \mathcal{K}$.

 %Now for any Borel set $A \subset X$ we have by Jensen's inequality
 %\begin{align*}
 % \mu(A)\log\frac{\mu(A)}{\hat\mm(A)} & \le \int_A \hat\rho\log\hat\rho\,\d\hat\mm
  %      = \int_X \hat\rho\log\hat\rho\,\d\hat\mm - \int_{X\setminus A} \hat\rho\log\hat\rho\,\d\hat\mm \\  
  %      & = \entv(\mu) + \int_Xc_1\Wgh^2(x)\rho\,\d\mm + c_2 - \int_{X\setminus A} \hat\rho\log\hat\rho\,\d\hat\mm \\
  %      & \le C + c_2 + \frac{\hat\mm(X\setminus E)}{e} \le C + c_2 + \frac{1}{e}
 %\end{align*}
 %and so the set $\mathcal{K}$ is tight and the minimizers exist.
\end{proof}

As a technical tool we will need the excess mass functional
$\mathcal{F}_C \colon \ProbabilitiesTwo X \to [0,1]$ which is defined 
for all thresholds $C \ge 0$ as
\begin{equation}\label{eq:excessmass}
 \mathcal{F}_C(\mu) = \|(\rho-C)^+\|_{L^1(X,\mm)} + \mu^s(X),
\end{equation}
where $\mu = \rho \mm + \mu^s$ with $\mu^s \perp \mm$. This functional,
lower semicontinuous under weak convergence, 
was used in \cite{R2011b} to obtain the first good geodesics in 
$CD(K,N)$ spaces. The motivation for using the excess mass functional 
is that its variations under perturbation of the minimizer are easier to estimate,
since one only cares about the amount of mass exceeding the threshold.

\subsection{Localization in transport distance}\label{ss42}

As we will later see, the task of finding the first good intermediate
measure between $\mu_0$ and $\mu_1$ is slightly more difficult than 
finding the rest of the geodesic. This is due to the fact that after 
some $\mu_t$ with $t \in (0,1)$ has been fixed we can consider the 
transport distances to be essentially constant. This useful 
observation was made by Rajala in \cite{R2012}. It follows from two simple 
statements. First when one fixes an intermediate measure, 
the length of the curves along which the transport is done gets 
fixed. This is the content of the next proposition which was proved 
in \cite[Proposition 1]{R2012}.

\begin{proposition}\label{prop:separation}
 Let $\mu_0,\, \mu_1 \in  \ProbabilitiesTwo X$ and $t_0 \in (0,1)$.
 Suppose that there exist constants $0 \le C_1 \le C_2 < \infty$ 
 and a measure  $\ppi \in \gopt(\mu_0,\mu_1)$ with
 \begin{equation}\label{eq:lengthbounds}
  C_1 \le l(\gamma) \le C_2 \qquad\text{for }\ppi\text{-a.e. }\gamma \in \geo(X).
 \end{equation}
 Then the bounds in \eqref{eq:lengthbounds} hold $\tilde\ppi$-a.e. 
 for any $\tilde\ppi \in \gopt(\mu_0,\mu_1)$ with 
 $(\e_{t_0})_\sharp\tilde\ppi = (\e_{t_0})_\sharp\ppi$.
\end{proposition}

In order to use the previous proposition we will need another 
observation which is a simple consequence of cyclical monotonicity
(cf. Chapter 5 in Villani's survey \cite{Villani09} for a review of cyclical monotonicity).
Namely, when we work on a part of the transport with some bounds on 
the lengths of the curves, this part will not get mixed with other 
parts of the measure at any intermediate time. For the proof of this 
fact see \cite[Lemma 2.5]{R2012}.

\begin{lemma}\label{lma:separation}
 Take $0 \le C_1\le C_2 \le C_3 \le C_4 \le \infty$ and define 
 \[
  A_1 = \{\gamma \in \geo(X) \,:\, C_1 \le l(\gamma) \le C_2 \}
 \quad \text{and} \quad
  A_2 = \{\gamma \in \geo(X) \,:\, C_3 < l(\gamma) \le C_4 \}.
 \]
 Then for any $\ppi \in \gopt(\mu_0,\mu_1)$ and any $t \in (0,1)$
 there exists a Borel set $E \subset \geo(X)$ with $\ppi(E)=0$ such that
 \[
  \{(\gamma, \hat\gamma) \in (A_1\setminus E) \times (A_2\setminus E)\,:\,
  \gamma_t = \hat\gamma_t\} = \emptyset.
 \]
\end{lemma}

\subsection{Density bounds for the minimizers}\label{ss43}

The information from the minimizers of the entropy and of the excess 
mass functional are obtained with a contradiction argument. First we 
assume that there exists a minimizer which does not have the desired 
density bound. After this we isolate the part of the minimizer where 
the density bound is exceeded and redefine this part of the measure to 
be something slightly better. If this new measure is again an 
intermediate point and we have strictly decreased the energy we are minimizing (the entropy or the excess mass)
we obtain a contradiction, so that the minimizer must satisfy the density bound. To prove that we indeed get an intermediate point we use 
the next lemma, whose proof relies on the joint convexity of $(\mu,\nu)\mapsto W_2^2(\mu,\nu)$, 
which was again proved by Rajala in \cite[Lemma 3.5]{R2011b}.

\begin{lemma}\label{lma:combined}
 Let $\mu_0,\, \mu_1 \in \ProbabilitiesTwo X$. Then for any $\lambda\in (0,1)$, 
 any $\ppi \in \gopt(\mu_0,\mu_1)$, any Borel function 
 $f \colon \geo(X) \to [0,1]$ with $c = (f\ppi)(\geo(X)) \in (0,1)$
 and any 
 \[
  \nu \in \mathcal{I}_\lambda\left(\frac1{c} (\e_{0})_\sharp\left(f\ppi\right),
      \frac1{c} (\e_{1})_\sharp\left(f\ppi\right)\right)
 \]             
 we have
 \[
  (\e_{\lambda})_\sharp\left((1-f)\ppi\right) + c\nu \in \mathcal{I}_\lambda(\mu_0, \mu_1).
 \]
\end{lemma}

The first step which uses the minimization of the excess mass functional $\mathcal{F}_C$
 in \eqref{eq:excessmass} is the same one 
that was taken in \cite[Proposition 3.11]{R2011b}. We repeat some key 
points of the proof for the convenience of the reader. In \cite{R2011b} 
the functionals $\mathcal{F}_C$ were minimized only in the bounded case. 
A reduction to this case can be also made here and so the following 
proposition which was proved in a slightly different form in 
\cite[Proposition 3.9 and Proposition 3.11]{R2011b} will suffice.

\begin{proposition}\label{prop:excesszero}
 Assume that $(X,\sfd)$ is a bounded metric space with a finite measure $\mm$.
 Let $\nu_0, \,\nu_1 \in \ProbabilitiesTwo X$ and $t \in [0,1]$. Suppose 
 that there exists a constant $C> 0$ so that for any
 $\ppi \in \gopt(\nu_0,\nu_1)$ and $A\subset X$  Borel with
 $\ppi(\e_t^{-1}(A))> 0$ we have that for the measures
 \begin{equation}\label{eq:hatdef}
  \hat\nu_0 = \frac{1}{\ppi(\e_t^{-1}(A))}(\e_0)_\sharp\left(\ppi\res\e_t^{-1}(A)\right), \qquad
  \hat\nu_1 = \frac{1}{\ppi(\e_t^{-1}(A))}(\e_1)_\sharp\left(\ppi\res\e_t^{-1}(A)\right)
 \end{equation}
 there exists a measure $\hat\nu \in \mathcal{I}_t(\hat\nu_0,\hat\nu_1)$ 
 with
 \begin{equation}\label{eq:entropybound1}
  \entv(\hat\nu) \le \log\frac{C}{\ppi(\e_t^{-1}(A))}.
 \end{equation}
 Then there exists a minimizer $\mu_t$ of $\mathcal{F}_{C}$ in
 $\mathcal{I}_t(\nu_0,\nu_1)$ and the minimum value is zero, so that
 $\mu_t\ll\mm$ and its density is less than $C$ $\mm$-a.e. in $X$.
\end{proposition}

\begin{proof} 
 Take a threshold $C' > C$. It suffices to prove that the minimum of
 $\mathcal{F}_{C'}$ in $\mathcal{I}_t(\nu_0,\nu_1)$ is zero and then
 let $C' \downarrow C$. Without loss of generality we may assume that
 all minimizers, whose existence is ensured by tightness of $\mathcal{I}_t(\nu_0,\nu_1)$
 in $\prob{X}$ and lower semicontinuity,
  are absolutely continuous with respect to $\mm$.
 Indeed, suppose that there is a measure $\omega\in \mathcal{I}_t(\nu_0,\nu_1)$ with
a singular part. Let $A$ be an $\mm$-negligible Borel set where
the singular part of $\omega$ is concentrated. By the assumption of the Proposition
together with Lemma~\ref{lma:combined} we can then redefine the part of
$\omega$ which is supported on $A$ to be a measure having finite 
entropy. In particular it will be absolutely continuous with respect to
$\mm$. Since we are redefining only the singular part of $\omega$,
the value of the functional $\mathcal{F}_{C'}$ does not increase after
the redefinition.

 Assume, contrary to the claim, that the infimum of $\mathcal{F}_{C'}$
 in $\mathcal{I}_t(\nu_0,\nu_1)$ is positive. Denote by
 $\mathcal{M}_\text{min} \subset \mathcal{I}_t(\nu_0,\nu_1)$ the set of
 minimizers of $\mathcal{F}_{C'}$ in $\mathcal{I}_t(\nu_0,\nu_1)$.
 Applying the proof of \cite[Proposition 3.9]{R2011b} we see that
 the set $\mathcal{M}_\text{min}$ is always nonempty.
 Take $\nu \in \mathcal{M}_\text{min}$  for which
 \begin{equation}\label{eq:almostmaxpos}
  \mm(\{x \in X ~:~ \rho_\nu(x) > C'\}) \ge \left(\frac{C}{C'}\right)^{\frac14} 
  \sup_{\omega \in \mathcal{M}_\text{min}}\mm(\{x \in X ~:~ \rho_\omega(x) > C'\}),
 \end{equation}
 where $\nu = \rho_\nu\mm$ and $\omega = \rho_\omega \mm$. Let
 $\ppi \in \gopt(\nu_0,\nu_1)$ be such that $(\e_t)_\sharp\ppi = \nu$.

 There exists $\delta >0$ so that 
 \[
  \mm(A) > \left(\frac{C}{C'}\right)^{\frac12} \mm(A')  
 \]
 with
 \begin{equation}\label{eq:deltabound}
  A' = \{x \in X ~:~\rho_\nu(x) > C'\} \quad \text{ and }\quad
  A = \{x \in A' ~:~ \rho_\nu(x) > C' + \delta\}.  
 \end{equation}

 {F}rom the assumption of the proposition we know the existence of
 a measure $\hat\nu=\hat{\rho}\mm\in \mathcal{I}_t(\hat\nu_0,\hat\nu_1)$ with
 $\entv(\hat\nu) \le \log ({C}/{\nu(A)})$,
 where $\hat\nu_0$ and $\hat\nu_1$ are given by \eqref{eq:hatdef}.
 By Jensen's inequality we then have
 \begin{equation}\label{eq:bigsupport}
  \mm(\{\hat\rho>0\}) \ge \frac{\nu(A)}{C} \ge \frac{C'}{C} \mm(A)
         \ge \left(\frac{C'}{C}\right)^{\frac12} \mm(A').
 \end{equation}
 
 We can now consider a new measure $\tilde\nu = \tilde\rho \mm$ defined
 as the combination
 \begin{equation}\label{eq:tildenu}
  \tilde\nu = \nu\res(X \setminus A) + \frac{C'}{C'+ \delta} \nu\res A
               + \frac{\delta}{C'+\delta}\nu(A) \hat\nu.
 \end{equation}
 By Lemma~\ref{lma:combined} and the convexity of $\mathcal{I}_t$ we have
 $\tilde\nu \in \mathcal{I}_t(\nu_0,\nu_1)$. Due to the definition
 \eqref{eq:deltabound} we only redistribute some of the mass above
 the density $C'$ when we replace the measure $\nu$ by the measure
 $\tilde\nu$, so that $\tilde\nu\in \mathcal{M}_{\text{min}}$. Let us calculate how much the excess mass functional
 changes in this replacement:
 \[
  \mathcal{F}_{C'}(\nu) - \mathcal{F}_{C'}(\tilde\nu) 
   =\int_{\{\rho_\nu < C'\}}\min\left\{C'-\rho_\nu, \frac{\delta}{C'+\delta}\nu(A)\hat\rho\right\}\,\d\mm.
 \]
 Because of the minimality of $\mathcal{F}_{C'}$ at $\nu$ this
 integral must be zero. Therefore $\{\hat\rho>0\}\cap\{\rho_\nu<C'\}$
 is $\mm$-negligible. On the other hand, for any
 $y\in \{\hat\rho>0\}\cap\{ \rho_\nu \ge C'\}$ we have
 $\tilde\rho(y) > C'$ (if $y\in X\setminus A$ this is trivial, if $y\in A$ the second term
 in \eqref{eq:tildenu} gives a contribution larger than $C'$). This, together with our choice 
 \eqref{eq:almostmaxpos} of $\nu$, leads to a contradiction:
 $$
  \mm(\{\tilde\rho > C'\}) \geq  \mm(\{\hat\rho>0\})
              \ge \left(\frac{C'}{C}\right)^{\frac12} \mm(A') 
     \ge \left(\frac{C'}{C}\right)^{\frac14} \sup_{\omega \in \mathcal{M}_\text{min}}\mm(\{\rho_\omega> C'\}).
 $$
\end{proof}

 Next we make another minimization. This time for the entropy itself.
 A similar argument was used in \cite{R2012} to obtain good geodesics
 in metric spaces satisfying the reduced curvature dimension 
 condition $CD^*(K,N)$.

\begin{proposition}\label{prop:spreadtolimit}
 Let $\mu_0, \,\mu_1 \in \ProbabilitiesTwo X$ and $t \in [0,1]$.
 Suppose that there exists a constant $C> 0$ so that for any
 $\ppi \in \gopt(\mu_0,\mu_1)$ and $A \subset X$ Borel with
 $\ppi(\e_t^{-1}(A))> 0$ we have that for the restricted measures
 $\hat{\nu}_0,\,\hat{\nu}_1$ in \eqref{eq:hatdef} 
 there exists a measure
 $\hat\nu \in \mathcal{I}_t(\hat\mu_0,\hat\mu_1)$ satisfying \eqref{eq:entropybound1}.
 Then for any minimizer $\mu_\text{min}$ of the entropy in
 $\mathcal{I}_t(\mu_0,\mu_1)$ we have $\mu_\text{min}\leq C\mm$.
\end{proposition}

\begin{proof} Without loss of generality, we can assume $t\in (0,1)$.
 Let $\nu=\rho\mm$ be one of the minimizers of the entropy in
 $\mathcal{I}_t(\mu_0,\mu_1)$, which by Lemma~\ref{lma:minexists} we
 know to exist. By \eqref{eq:entropybound1} with $A=X$ we know that $\entv(\nu)<\infty$.
 We need only to show that $\mathcal F_C(\nu)=0$.
 
 Let $\ppi \in \gopt(\mu_0,\mu_1)$ be such that
 $(\e_t)_\sharp\ppi = \nu$.   
 Suppose now by contradiction that $\mathcal{F}_C(\nu) > 0$, let
 $\eta > 0$ be such that $\mm(\{\rho > C + 2\eta\}) > 0$ and 
 define
 \[
  C_1 = \frac{1}{\eta}\left[\mm(\{\rho > C + \eta\})
          -\mm(\{\rho > C + 2\eta\})\right] \ge 0.
 \]
 Since $\tau\mapsto g(\tau):=\mm(\{\rho \geq C + \tau\})$ is nonincreasing,
 there exists $\delta \in (\eta,2\eta)$ such that $-g'(\delta)\leq C_1$. In particular, choosing $\delta$ in this way
 and fixing $x_0\in X$, for $\phi\in (0,\eta/3)$ sufficiently small and $R=R(\phi)$ sufficiently large one has
 $\mm(L') < \mm(L) + (1+C_1)\phi$, where
 \[
  L = \{x \in B(x_0,R) \,:\, \rho(x) > C + \delta\}\qquad 
   \text{and}\qquad L' = \{x \in X \,:\, \rho(x) \ge C + \delta - 3\phi\}.
 \]

 Let $\Gamma \subset \geo(X)$ be a cyclically monotone set on which $\ppi$ is supported.
 Fix $\bar\gamma \in \Gamma \cap \rme_t^{-1}(L)$ and consider any $\gamma \in \Gamma \cap \rme_t^{-1}(L)$.
 Using cyclical monotonicity we get (similarly as in \cite[Theorem 8.22]{Villani09}) 
 \begin{align*}
  \sfd^2(\gamma_0, \gamma_1) & \le \sfd^2(\bar\gamma_0, \bar\gamma_1) + \sfd^2(\gamma_0, \gamma_1) 
                    \le \sfd^2(\bar\gamma_0, \gamma_1) + \sfd^2(\gamma_0, \bar\gamma_1)\\
                    & \le \left(\sfd(\gamma_t, \gamma_1) + \diam(L) + l(\bar\gamma)\right)^2 + \left(\sfd(\gamma_0, \gamma_t) + \diam(L) + l(\bar\gamma)\right)^2\\
                    & = \left((1-t)^2 + t^2\right)\sfd^2(\gamma_0, \gamma_1) + 2(\diam(L) + l(\bar\gamma))\sfd(\gamma_0, \gamma_1) + 2(\diam(L) + l(\bar\gamma))^2.
 \end{align*}
 Since $(1-t)^2 + t^2 = 1 - 2(1-t)t < 1$, the length of the geodesic $\gamma$ has a bound from above
 given in terms of only $\diam(L)$ and $l(\bar\gamma)$.
 Hence the measure $\ppi\res\rme_t^{-1}(L)$ is supported in a uniformly bounded set of curves.

 We can use Proposition~\ref{prop:excesszero} with $\nu_i=(\nu(L))^{-1}(\e_i)_\sharp\ppi\res\e_t^{-1}(L)$
 to find a measure
 \[
  \tilde\nu = \tilde\rho \mm \in \mathcal{I}_t\left(
  \frac{(\e_0)_\sharp\ppi\res\e_t^{-1}(L)}{\nu(L)},
  \frac{(\e_1)_\sharp\ppi\res\e_t^{-1}(L)}{\nu(L)}
  \right)
 \]
 with $\tilde\rho\le {C}/{\nu(L)}$ $\mm$-a.e. in $X$.

 Now consider a new measure $\hat\nu = \hat\rho \mm$ defined as the
  combination
 \[
  \hat\nu = \nu\res (X \setminus L)
            + \frac{C + \delta-\phi}{C+ \delta} \nu\res L
            + \frac{\phi}{C+\delta}\nu(L) \tilde\nu.
 \]
 By Lemma~\ref{lma:combined} we have
 $\hat\nu \in \mathcal{I}_t(\mu_0,\mu_1)$.

 For $x \in L$ we have the estimates
 \begin{align}
  \hat\rho(x) & \le \frac{C + \delta -\phi}{C+ \delta}\rho(x)
                    + \frac{\phi}{C+\delta}\nu(L)\tilde\rho(x)
                \le \frac{(C + \delta-\phi)\rho(x) + C\phi}{C+\delta} \label{eq:tapio1}\\
              & = \rho(x) + \frac{(C-\rho(x))\phi}{C + \delta}
                < \rho(x) - \frac{\delta\phi}{C + \delta}\nonumber
 \end{align}
 and
 \begin{equation}\label{eq:tapio2}
   \hat\rho(x) \ge \frac{C + \delta - \phi}{C+ \delta}\rho(x)
                 > C + \delta - \phi.
 \end{equation}
 For $x \in L' \setminus L$ we have
 \begin{equation}\label{eq:tapio3}
  \hat\rho(x) \le \rho(x) + \frac{\phi}{C+\delta}\nu(L)\tilde\rho(x)
          \le \rho(x) + \frac{C\phi}{C + \delta} < C + \delta + \phi
 \end{equation}
 and for $x \in X \setminus L'$ we get
 \begin{equation}\label{eq:tapio4}
  \hat\rho(x) \le \rho(x) + \frac{\phi}{C+\delta}\nu(L)\tilde\rho(x)
              \le C + \delta - 3\phi + \frac{C\phi}{C + \delta} 
              < C + \delta - 2\phi.
 \end{equation}

 Write $C_2 = \frac{\delta}{C + \delta}\mm(L)$. Let us estimate
 the change in the entropy when we replace $\nu$ by $\hat\nu$: 
 using the convexity inequality $x\log x-y\log y\leq (x-y)(\log x+1)$   
  we can estimate from above $\entv(\hat\nu)-\entv(\nu)$ by
 $$
 \int_X(\hat\rho- \rho) (\log\hat\rho+1)\,\d\mm 
     %+ \int_{X \setminus L}(\hat\rho - \rho)(\log\rho + 1)\,\d\mm 
   =\int_X(\hat\rho- \rho) \log\hat\rho\,\d\mm .
   %+ \int_{X \setminus L}(\hat\rho - \rho)\log\rho\,\d\mm.
 $$
 Now, we set $w:=\hat\rho-\rho$, split $X$ as $L\cup (X\setminus L')\cup (L'\setminus L)$ 
 and use the fact that $w\leq 0$ on $L$ and $w\geq 0$ on $X \setminus L$, the inequalities
 \eqref{eq:tapio1}, \eqref{eq:tapio2}, \eqref{eq:tapio3},
 \eqref{eq:tapio4} and eventually the concavity of $\log$ to get
 \begin{align*}
   &~ \int_L w \log\left(C + \delta-\phi\right)\,\d\mm 
     + \int_{X \setminus L'} w\log\left(C + \delta-2\phi\right)\,\d\mm 
     + \int_{L' \setminus L} w\log\left(C + \delta+\phi\right)\,\d\mm \\
   = &~\left(\log\left(C + \delta-\phi\right) -  \log\left(C + \delta - 2\phi\right)\right)\int_L w\,\d\mm  
    + \left(\log\left(C + \delta+\phi\right) - \log\left(C + \delta - 2\phi\right)\right) \int_{L'\setminus L} w\,\d\mm \\
   \le &~-\left(\log\left(C + \delta-\phi\right) -  \log\left(C + \delta - 2\phi\right)\right) \frac{\delta\phi}{C + \delta}\mm(L)  \\
   & + \left(\log\left(C + \delta+\phi\right) - \log\left(C + \delta - 2\phi\right)\right) \frac{C\phi}{C + \delta}\mm(L'\setminus L) \\
   < &~-\left(\log\left(C + \delta-\phi\right) -  \log\left(C + \delta - 2\phi\right)\right) C_2\phi  
   + \left(\log\left(C + \delta+\phi\right) - \log\left(C + \delta - 2\phi\right)\right) (1+C_1)\phi^2 \\
   \le &- C_2\phi\frac{\phi}{C + \delta - 2\phi}+(1+C_1)\phi^2\frac{3\phi}{C + \delta - 2\phi} < 0
 \end{align*}
  for small enough $\phi\in (0,\eta/3)$. This contradicts the minimality of
 the entropy at $\nu$.
\end{proof}

\subsection{Construction of the geodesic}\label{ss44}

\begin{proof}[Proof of Theorem \ref{thm:goodgeodesics}]
 In this proof, to avoid a cumbersome notation, we switch to the ${\rm exp}$ notation
 and set $C_1:=\|\rho_1\|_{L^\infty(X,\mm)}$.
 Let $D > 0$ be such that $\supp(\mu_1) \subset B(x_0,D)$.
 We will prove the claim with $$t_0 := \min\{\frac{c_2}{2K^-},\frac12\}.$$
 The geodesic is constructed as follows.  First we fix the 
 measure $\mu_{t_0} =\rho_{t_0}\mm\in \mathcal{I}_{t_0}(\mu_0,\mu_1)$ to be a
 minimizer of the entropy in $\mathcal{I}_{t_0}(\mu_0,\mu_1)$.
 After this we define the rest of the geodesic for times
 $t \in (0,t_0)$ inductively. Suppose that for some $n \in \N$
 we have defined $\mu_{k2^{-n}t_0}$ for all $k = 0, 1, \ldots, 2^n$.
 Then for all odd $k \in \N$ with $0 < k < 2^{n+1}$ we define
 $\mu_{k2^{-n-1}t_0}$ to be a minimizer of the entropy in
 $\mathcal{I}_{\frac12}(\mu_{(k-1)2^{-n-1}t_0},\mu_{(k+1)2^{-n-1}t_0})$.
 We construct the geodesic on the interval $(t_0,1]$ in a similar way by 
 iteratively selecting the midpoints with minimal entropy.
 The rest of the geodesic is given by completion. Let
 $\ppi \in \gopt(\mu_0,\mu_1)$ be such that $(\e_t)_\sharp\ppi = \mu_t$
 for all $t \in [0,1]$.
 
 Since we are selecting minimizers of the entropy among all the
 possible intermediate measures in a $CD(K,\infty)$-space, the
 selected measures satisfy the convexity inequality
 \eqref{eq:CDdef} between the given endpoint measures.
 Therefore, by Proposition~\ref{prop:combinedgeod} the inequality
 \eqref{eq:CDdef} holds for all $t \in [0,1]$.

 Let us then concentrate on the entropy estimates assumed in
 Proposition~\ref{prop:excesszero} and Proposition~\ref{prop:spreadtolimit}. 
 Let $\ppi \in \gopt(\mu_0,\mu_1)$ and
 $A\subset X$ Borel with $M := \ppi(\e_{t_0}^{-1}(A))>0$, write
 \[
  \hat\mu_0 = \hat\rho_0\mm = \frac{1}{M}(\e_0)_\sharp\left(\ppi\res\e_{t_0}^{-1}(A)\right) \quad\text{and}\quad
  \hat\mu_1 = \hat\rho_1\mm = \frac{1}{M}(\e_1)_\sharp\left(\ppi\res\e_{t_0}^{-1}(A)\right),
 \]
 and take a measure $\nu \in \mathcal{I}_{t_0}\left(\hat\mu_0, \hat\mu_1\right)$
 which satisfies the convexity inequality \eqref{eq:CDdef}
 between these measures. Now, using \eqref{eq:decay}, 
 we have the estimate (with $\Wgh(x)=\sfd(x,x_0)$)
 \begin{align*}
  \entv(\nu)
  \le~& (1-t_0)\entv\left(\hat\mu_0\right) + t_0 \entv\left(\hat\mu_1\right)
                      + \frac{K^-}{2}t_0(1-t_0)W_2^2\left(\hat\mu_0, \hat\mu_1\right)\\
                  \le~& t_0\log\left(\frac{C_1}{M}\right)
                     + (1-t_0)\int_X \hat\rho_0(x)\left(\log\hat\rho_0(x) + \frac{K^-}2t_0(D+\Wgh(x))^2\right)\,\d\mm(x) \\
                  \le~& t_0\log\left(\frac{C_1}{M}\right)
                     + (1-t_0)\int_X \hat\rho_0(x)\left(\log\left(\frac{c_1}M\right) - c_2\Wgh^2(x) + K^-t_0(D^2+\Wgh^2(x))\right)\,\d\mm(x) \\
                  \le~& \log\left(\frac{\max\{C_1,c_1\}}{M}\right) + K^-D^2
                  = \log\left(\frac{\max\{C_1,c_1\}{\rm exp}[K^-D^2]}{M}\right),
 \end{align*}
 since $K^-t_0 \le c_2$ by the choice of $t_0$. By 
 Proposition~\ref{prop:spreadtolimit} we then have the estimate
 \[
  \|\rho_{t_0}\|_{L^\infty(X,\mm)} \le \max\{C_1,c_1\}{\rm exp}[K^-D^2]
  \le \max\{C_1,c_1\}{\rm exp}[(2K^-+c_2)D^2]=: C.
 \]

 Next we prove that for all $t \in [0,t_0]$ we have
 $\mu_t = \rho_t\mm$ with the estimate
 \begin{equation}\label{eq:decayestimate}
  \rho_t(\gamma_t) \le C{\rm exp}\bigl[-\frac12(1-\frac{t}{t_0})(c_2 - K^-tt_0)\ell^2(\gamma)\bigr] \quad \text{for $\ppi$-a.e. $\gamma \in \geo(X)$.}
 \end{equation}
 First of all the estimate \eqref{eq:decayestimate} is true
 for $t = t_0$. For $t = 0$ we have that, thanks to \eqref{eq:decay}, $\rho_0(\gamma_0)$ can be estimated from above by
 \[
  c_1{\rm exp}\bigl[-c_2\sfd^2(\gamma_0,x_0)\bigr] \le c_1{\rm exp}\bigl[-c_2([\ell(\gamma)-D]^+)^2\bigr] 
                   \le c_1{\rm exp}\bigl[-\frac{c_2}2\ell^2(\gamma)+c_2D^2\bigr] \le C{\rm exp}(-\frac{c_2}2\ell^2(\gamma))
 \]
 and so \eqref{eq:decayestimate} holds also at $t=0$.
 
 Suppose that for some $n \in \N$ the estimate
 \eqref{eq:decayestimate} holds for all $t = k2^{-n}t_0$
 with $k = 0,1,\dots, 2^n$. Take an odd integer $k$ with
 $0 < k < 2^{n+1}$. Our aim is to prove \eqref{eq:decayestimate}
 for $t = k2^{-n-1}t_0$.

 Let $l \in (0,\infty)$ and $\epsilon > 0$ be such that we have
 $\tilde M = \ppi(\{\gamma\,:\,l \le l(\gamma) \le l+\epsilon\})>0$. 
 Then by Proposition~\ref{prop:separation} we know that any measure 
 \[
  \tilde\ppi \in \gopt\left(\frac1{\tilde M}(\e_0)_\sharp\ppi\res{\{\gamma\,:\,l \le \ell(\gamma) \le l+\epsilon\}},
  \frac1{\tilde M}(\e_1)_\sharp\ppi\res{\{\gamma\,:\,l \le \ell(\gamma) \le l+\epsilon\}}\right)  
 \]
 is concentrated on geodesics with lengths in the interval
 $[l,l+\epsilon]$. On the other hand, by Lemma~\ref{lma:separation}
 we know that
 \[
  (\e_{k2^{-n-1}t_0})_\sharp\tilde\ppi \perp (\e_{k2^{-n-1}t_0})_\sharp
  \ppi\res\{\gamma\,:\, \ell(\gamma) \notin [l,l+\epsilon] \text{ and }\gamma_{k2^{-n-1}t_0} \in A\}.
 \]

 Therefore, in proving \eqref{eq:decayestimate} we may separately
 deal with the parts of the measure where all the geodesics have
 lengths in an interval $[l,l+\epsilon]$. Take now
 a Borel set $A\subset X$ such that for the measure
 $\hat\ppi = \ppi\res\{\gamma\,:\,l \le \ell(\gamma) \le l+\epsilon\text{ and }\gamma_{k2^{-n-1}t_0} \in A\}$
 we have $\hat{M} = \hat\ppi(\geo(X))>0$.

 Suppose that the measure 
 \[
  \tilde\nu \in \mathcal{I}_\frac12\left(\frac1{\hat{M}}(\e_{(k-1)2^{-n-1}t_0})_\sharp\hat\ppi,
  \frac1{\hat{M}}(\e_{(k+1)2^{-n-1}t_0})_\sharp\hat\ppi\right)  
 \]
 satisfies the convexity inequality \eqref{eq:CDdef}. Then
 \begin{align*}
  \entv(\tilde\nu) \le ~& \frac12\entv(\hat{M}^{-1}(\e_{(k-1)2^{-n-1}t_0})_\sharp\hat\ppi) + \frac12\entv(\hat{M}^{-1}(\e_{(k+1)2^{-n-1}t_0})_\sharp\hat\ppi)\\
   & + \frac{K^-}{8}W_2^2\left(\hat{M}^{-1}(\e_{(k-1)2^{-n-1}t_0})_\sharp\hat\ppi,\hat{M}^{-1}(\e_{(k+1)2^{-n-1}t_0})_\sharp\hat\ppi\right)\\
   \le ~& \frac12 \log{\frac{C}{\hat{M}}} -\frac14((1-(k-1)2^{-n-1})(c_2 - K^-(k-1)2^{-n-1}t_0^2)l^2) \\
   &+ \frac12 \log{\frac{C}{\hat{M}}} -\frac14((1-(k+1)2^{-n-1})(c_2 - K^-(k+1)2^{-n-1}t_0^2)l^2) \\
    & + \frac{K^-}{8}(2^{-n}t_0(l+\epsilon))^2\\
   = ~& \log{\frac{C}{\hat{M}}} - \frac12((1-k2^{-n-1})(c_2 - K^-k2^{-n-1}t_0^2)l^2) + \frac{K^-}{8}2^{-2n}t_0^2(2l+\epsilon)\epsilon.
 \end{align*}
 Proposition \ref{prop:spreadtolimit} then gives
 \[
  \rho_t(\gamma_t) \le C{\rm exp}\bigl[-\frac12(1-\frac{t}{t_0})(c_2 - K^-tt_0)l^2 + \frac{K^-}{8}2^{-2n}t_0^2(2l+\epsilon)\epsilon\bigr]
 \]
 for $\ppi$-a.e. $\gamma \in \geo(X)$ with
 $\ell(\gamma) \in [l,l+\epsilon]$. By letting $\epsilon \downarrow 0$ we then obtain
 \eqref{eq:decayestimate} for $t = k2^{-n-1}t_0$.

 Notice that the estimate \eqref{eq:decayestimate} gives
 $ \rho_t(\gamma_t) \le C{\rm exp}\bigl[-\frac12(1-\frac{t}{t_0})(c_2 - K^-tt_0)\ell^2(\gamma)\bigr] \le C$
 for all $t \in [0,t_0]$ for $\ppi$-a.e $\gamma \in \geo(X)$,
 which is equivalent to \eqref{eq:uniformbound}.
\end{proof}

\section{Convergence results}\label{sec:auxiliary}

This section is devoted to the proof of some auxiliary convergence results. The first one deals with
entropy convergence. Recall the notation $\Wgh(x)=\sfd(x,x_0)$.

\begin{lemma}\label{lem:ConvEnt}
Let $f_n\mm$, $f\mm$ be positive finite measures in $X$. If $f_n \up f$ $\mm$-a.e. and 
$\int f\Wgh^2\,\d\mm<\infty$, then 
\begin{equation}\label{eq:ConvEnt}
\int_X f_n \log f_n \,\d \mm \to \int_X f\log f\, \d \mm.
\end{equation}
The same conclusion holds if  $f_n \down f$ $\mm$-a.e. and $\int f_1\Wgh^2\,\d\mm<\infty$.
\end{lemma}
\begin{proof} Assume first that $\mm$ is a finite measure. Let us first consider the case $f_n \up f$. 
Observe that the function $t\mapsto t \log t$  is decreasing on $[0,\rme^{-1}]$ and  increasing on $[\rme^{-1},\infty)$; we write it as the difference
$\phi_1-\phi_2$, with
$$
\phi_1(t):=
\begin{cases}
-\frac{1}{\rme} &\text{if $t\in [0,\frac{1}{\rme}]$;}\\
t\log t & \text{if $t\geq\frac{1}{\rme}$,}
\end{cases}\qquad\qquad
\phi_2(t):=
\begin{cases}
-\frac{1}{\rme}-t\log t &\text{if $t\in [0,\frac{1}{\rme}]$;}\\
0 & \text{if $t\geq\frac{1}{\rme}$.}
\end{cases}
$$
Notice that $\phi_i$ are nondecreasing and bounded from below. Therefore
we can apply the monotone convergence theorem for 
$\int\phi_i(f_n)\,\d\mm$ to conclude. In the case $f_n\down f$ the argument is the same.

In the general $\sigma$-finite case we use
\eqref{eq:changeentropy} to reduce ourselves to the previous case, noticing that our assumptions
on $f_n$ imply $\int f_n\Wgh^2\,\d\mm\to\int f\Wgh^2\,\d\mm<\infty$.
\end{proof}

Recall that, according to Definition~\ref{def:wug} and \eqref{eq:convention}, 
the space $\calW$ consists of $\mm$-measurable functions having a weak upper gradient in $L^2(X,\mm)$.

\begin{lemma}\label{lem4.4revised}
Let $x_0\in X$, $\mu=f\mm,\,\sigma=g\mm\in\probt{X}$ with $f(x)\leq c_1\rme^{-c_2\sfd^2(x,x_0)}$ for some $c_1,\,c_2>0$, 
$\inf_{B_R(x_0)}f>0$ for
all $R>0$ and $g\in L^\infty(X,\mm)$ with bounded support. 
Let $\ppi\in\gopt(\mu,\sigma)$ be a good geodesic given by Theorem~\ref{thm:goodgeodesics}. Then:
\begin{itemize}
\item[(1)] For $h\in\calW$ satisfying $\weakgrad{h}\in L^2(X,\mu)$ and
\begin{equation}\label{assNablaf}
\weakgrad{h}^2(x) \leq C(1+ \sfd^2(x,x_0)) \quad \text{for any $x\in B^c_{R_*}(x_0)$}
\end{equation} 
for some $C,\,R_*>0$, the following holds (understanding the integrals on $\geo(X)$)
\begin{equation}
\limsup_{t \down 0} \int\left| \frac{h(\gamma_t)-h(\gamma_0)}{\sfd(\gamma_t, \gamma_0)}\right|^2 \d \ppi(\gamma) \leq \int \weakgrad{h}^2(\gamma_0) \,
\d\ppi(\gamma).
\end{equation} 
\item[(2)]  For all Kantorovich potentials $\varphi$ relative to $(\mu,\sigma)$ with $|D \varphi|$ having linear growth
one has
\begin{equation}\label{derKant} 
\lim_{t\downarrow 0} \frac{\varphi(\gamma_0)-\varphi(\gamma_t)}{\sfd(\gamma_0, \gamma_t)}= \lim_{t\downarrow 0} \frac{\sfd(\gamma_0,\gamma_t)}{t}=\weakgrad{\varphi}(\gamma_0) \quad \text{in } L^2(C([0,1];X), \ppi).
\end{equation}
\end{itemize} 
\end{lemma}

\begin{proof} (1) Call $f_t$ the density of $(\e_t)_\sharp \ppi$, i.e. $(\e_t)_\sharp \ppi = f_t\mm$; we know that for $t>0$ sufficiently small,
say $t\in (0,t_0)$, $f_t$ exists and there exists a constant $C_*$ such that $f_t\leq C_*$ $\mm$-a.e. in $X$ for all $t\in (0,t_0)$.
By definition of weak upper gradient, for any $t \in (0,t_0)$ and $\ppi$-a.e. $\gamma$ one has
\begin{equation} \nonumber
\left| \frac{h(\gamma_t)-h(\gamma_0)}{\sfd(\gamma_t,\gamma_0)} \right|^2 \leq 
\frac{\left(\int_0^t \weakgrad{h}(\gamma_s)|\dot{\gamma}_s| \d s \right)^2 }
{\sfd ^2(\gamma_t,\gamma_0)} \leq \frac 1 t \int_0^t \weakgrad{h}^2(\gamma_s) \d s,
\end{equation}
therefore applying twice Fubini's theorem and using the identity $(\e_t)_\sharp\ppi=f_t\mm$ we get
\begin{equation}
\int \left| \frac{h(\gamma_t)-h(\gamma_0)}{\sfd(\gamma_t,\gamma_0)} \right|^2 \d \ppi(\gamma) \leq
 \int \left( \frac 1 t \int_0^t \weakgrad{h}^2(\gamma_s) \d s \right) \d \ppi(\gamma) 
 = \int_X \left(\frac{1}{t} \int_0^t f_s \d s \right) \weakgrad{h}^2 \,\d \mm.
\end{equation}
The conclusion of the lemma follows once  the following claim is proved: 
\begin{equation}\label{claim:f}
\lim_{t\down 0} \int_X \left(\frac{1}{t} \int_0^t f_s \d s \right) \weakgrad{h}^2\, \d \mm= \int_X \weakgrad{h}^2 f\,\d \mm.
\end{equation}

In order to prove the claim we use both the uniform $L^\infty$ estimates on $f_t$ and the
$2$-uniform integrability of $\Wgh^2$ w.r.t. $f_t\mm$. Notice first that the local boundedness of 
$f^{-1}$ implies $\weakgrad{h}^2\in L^1(B_R(x_0),\mm)$ for all $R>0$; moreover 
\begin{equation}\label{eq:sigmat}
\bar{f}_t:=\left(\frac{1}{t} \int_0^t f_s \d s \right)\to f \quad \text{in duality with } L^1(B_R(x_0),\mm).
\end{equation}
Indeed the weak convergence $f_t \mm \to f\mm$ implies the weak convergence of $\bar{f}_t$ to $f$ in the duality with
$C_b(B_R(x_0))$; then \eqref{eq:sigmat} follows by the uniform $L^\infty$ bound on $\bar{f}_t$.
Second, observe that \eqref{assNablaf} gives
\begin{eqnarray}
\left|\int_X \bar{f}_t \weakgrad{h}^2 \,\d \mm-\int_{B_R(x_0)} \bar{f}_t \weakgrad{h}^2\, \d \mm \right| &\leq&  
\frac{C}{t} \int_0^t\int_{B_R^c(x_0)} (1+\sfd^2(x,x_0)) f_s \,\d \mm\d s \label{ClaimPart1}\\
&& \to 0 \quad \text{as } R\to \infty \text{ uniformly in } t \in (0,t_0)\nonumber;
\end{eqnarray}
the second line comes from the observation that the geodesic $(f_s \mm)_{s\in [0,1]}$ is a compact subset in $(\probt X, W_2)$,
hence tight and 2-uniformly integrable (see \cite[Proposition~7.1.5]{Ambrosio-Gigli-Savare08}).
The claim \eqref{claim:f} follows then combining \eqref{ClaimPart1} and \eqref{eq:sigmat}.

\noindent
(2) Observe we are under the assumptions of the Metric Brenier Theorem 10.3 in \cite{Ambrosio-Gigli-Savare11}, therefore there exists
a Borel function $L$ satisfying
$L(\gamma_0):=\sfd(\gamma_0,\gamma_1)$ for $\ppi$-a.e. $\gamma \in \geo(X)$ and, in addition,
\begin{equation}\label{eq:metrbren}
\weakgrad{\varphi}(x)=|D^+ \varphi| (x)= L(x) \quad \text{for $\mm$-a.e. $x\in X$.}
\end{equation}
It trivially follows that for $\ppi$-a.e. $\gamma \in\geo(X)$
$$\weakgrad{\varphi}(\gamma_0)=\sfd(\gamma_0,\gamma_1)=\frac{\sfd(\gamma_0,\gamma_t)}{t} \quad
\text{for every $t\in (0,1)$.}$$ 
The missing part is the $L^2$ convergence of difference quotients, proved and stated in  \cite{Ambrosio-Gigli-Savare11} under
a different set of assumptions: we adapt the argument to our case, where $|D \varphi|$ has linear growth. 
Since by optimality we have for $\ppi$-a.e. $\gamma$ that
$$\varphi(\gamma_0)+\varphi^c(\gamma_1)=\frac{\sfd^2(\gamma_0,\gamma_1)}{2}, \qquad \varphi(\gamma_t)
+\varphi^c(\gamma_1)\leq \frac{\sfd^2(\gamma_t,\gamma_1)}{2}  ,$$
we get with a subtraction that
$$\varphi(\gamma_0)-\varphi(\gamma_t)\geq \frac{1-(1-t)^2 }{2}\sfd^2(\gamma_0,\gamma_1)
=\frac{2t-t^2}{2} \sfd^2(\gamma_0,\gamma_1) \quad \text{for $\ppi$-a.e. $\gamma$.}$$
Therefore, dividing both sides by $\sfd(\gamma_t,\gamma_0)=t\sfd(\gamma_1,\gamma_0)$, for $\ppi$-a.e. $\gamma$ one has
\begin{equation}\label{liminfKant}
\liminf_{t \down 0} \frac{\varphi(\gamma_0)-\varphi(\gamma_t)}{\sfd(\gamma_0,\gamma_t)}\geq
 \sfd(\gamma_0,\gamma_1)=\weakgrad{\varphi} (\gamma_0).
\end{equation}
On the other hand, by definition of ascending slope 
\begin{equation}\label{limsupKant}
\limsup_{t\down 0} \frac{\varphi(\gamma_0)-\varphi(\gamma_t)}{\sfd(\gamma_0,\gamma_t)}\leq 
|D^+ \varphi|(\gamma_0).
\end{equation}
So, combining \eqref{eq:metrbren} and \eqref{liminfKant} with \eqref{limsupKant} we get 
\begin{equation}\label{eq:ConvKanta.e.}
\lim_{t \down 0} \frac{\varphi(\gamma_0)-\varphi(\gamma_t)}{\sfd(\gamma_0,\gamma_t)} = \weakgrad{\varphi}(\gamma_0) \quad 
\text{for $\ppi$-a.e. $\gamma$.}
\end{equation}
Now we claim that 
\begin{equation}\label{eq:weakConvKant}
\frac{\varphi(\gamma_0)-\varphi(\gamma_t)}{\sfd(\gamma_0,\gamma_t)} \rightharpoonup \weakgrad{\varphi}\circ\e_0 
\qquad \text{weakly in } L^2(\geo(X),\ppi).
\end{equation}
Since by assumption $|D \varphi|$ has linear growth, by part (1) of the present lemma we have 
\begin{equation}\label{eq:part1present}
\limsup_{t\down 0}\int \left |\frac{\varphi(\gamma_0)-\varphi(\gamma_t)}{\sfd(\gamma_0,\gamma_t)}\right |^2\,\d\ppi\leq
\int \weakgrad{\varphi}^2(\gamma_0)\,\d\ppi. 
\end{equation}
If $\psi$ is a weak limit point of the difference quotients as $t\downarrow 0$, by Mazur's lemma a sequence of convex combinations of 
these difference quotients strongly converges in $L^2(\geo(X),\ppi)$ to $\psi$. Since a further subsequence converges $\ppi$-a.e., from 
\eqref{eq:ConvKanta.e.} we obtain that $\psi=|D^+ \varphi|$. By weak compactness, the claim follows.

We conclude by observing that the lower semicontinuity of the norm under weak convergence together with \eqref{eq:part1present} 
ensure convergence of the $L^2(\geo(X),\ppi)$ norms.
Since in Hilbert spaces weak convergence and convergence of the norms give strong convergence, the lemma is proved.
\end{proof}

Our third result deals with weak convergence in the weighted Cheeger space: it will be applied to sequences
of Kantorovich potentials. In this and in the next lemma we assume that $\C$ is quadratic, so that 
by Theorem~\ref{thm:weighted} $\C_\eta$ is quadratic whenever $\eta=g\mm\in\probt X$ with $g\in L^\infty(X,\mm)$ and with $\C(\sqrt{g})<\infty$.
Recall that $\mathcal E_\eta$ denotes, according to \eqref{eq:numeriamoanchequesta}, the bilinear form associated to
$\C_\eta$.

\begin{lemma}\label{lem:WeakConvK} 
Let $(X,\sfd,\mm)$ have a quadratic Cheeger energy.
Let $\eta=g\mm\in\ProbabilitiesTwo X$ with $g\in L^\infty(X,\mm)$ and $\C(\sqrt{g})<\infty$. 
Consider a sequence $(f_n)\subset\calW$ with
\begin{equation}\label{assfn}
\sup_{n \in \N} \int_X \weakgrad{f_n}^2\,\d\eta < \infty, \qquad  \sup_{n\in\N} |f_n|(x)\leq C (1+\sfd^2(x,x_0)),
\end{equation}
and assume that $f_n\to f$ $\mm$-a.e. in $X$. Then 
\begin{equation}
\lim_{n\to\infty} \mathcal E_{\eta} (f_n,\log g)= \mathcal E_{\eta}(f,\log g).
\end{equation}
\end{lemma}
\begin{proof} We argue as in Theorem~\ref{thm:weighted}. Let us consider the weighted measure
$$
\tilde\eta:=\frac{1}{1+\Wgh^2}\eta
$$
and the corresponding weighted Sobolev space $H:=L^2(X,\tilde\eta)\cap\calW_\eta$, endowed with the scalar product
$$\langle f,g\rangle_H:=\int_X fg \, \d\tilde\eta + \mathcal E_{\eta} (f,g).$$
Observe that, since $L^2(X,\tilde\eta)$ is a Hilbert space, in order to check the completeness of the norm $\|\cdot\|_H$ induced by this scalar product
 it is enough to check the lower semicontinuity of $\|\cdot\|_H$ with respect to strong convergence in $L^2(X,\tilde\eta)$;  but this is clear since ${\C} _\eta$ is lower semicontinuous with respect to $L^2(X,\eta)$ convergence and, on sequences uniformly bounded in $L^\infty(X,\eta)$, 
 the finiteness of $\eta$ turns $L^2(X,\tilde\eta)$ convergence into $L^2(X,\eta)$ convergence. By a truncation argument one obtains that ${\C}_\eta$ is
 $L^2(X,\tilde\eta)$-lower semicontinuous. We conclude that $(H,\langle\cdot,\cdot\rangle_H)$ is a Hilbert space 
 (it is even separable, see \cite[Proposition~4.10]{Ambrosio-Gigli-Savare11b}, but we shall not need this fact in the sequel). 

Now since $\eta\in \Probabilities X$, from the second assumption \eqref{assfn} and dominated convergence we have that 
$f_n \to f$ strongly in $L^2(X,\tilde\eta)$. On the other hand, the first assumption in \eqref{assfn} implies that $\|f_n\|_H$ is bounded. 
By reflexivity if follows that $f_n\to f$ weakly in $H$. The conclusion follows by noticing that, since $\C(\sqrt{g})<\infty$, the map
$$h\mapsto \mathcal E_\eta(h,\log g) $$
is linear and continuous from $H$ to $\R$.
\end{proof}

In this last result we estimate how much $\mathcal E_\rho(\log g,\varphi)$ changes under modifications of
the density $g$ of $\rho$.

\begin{lemma}\label{lem:essential} Let $\eta=g\mm,\,\eta'=g'\mm\in\probt{X}$ with $g,\,g'\in L^\infty(X,\mm)$ and $\C(\sqrt{g}),\,\C(\sqrt{g'})$ 
finite. Let 
$\varphi:X\to\R$ be a locally Lipschitz function whose gradient has linear growth. Then, setting $E:=\{g\neq g'\}$, one has
\begin{eqnarray}\label{eq:essential}
&&|\mathcal E_\eta(\log g,\varphi)-\mathcal E_{\eta'}(\log g',\varphi)|\\ &\leq&\biggl(\int_E\weakgrad{\sqrt{g}}^2\,\d\mm\biggr)^{1/2}
\biggl(\int_E \weakgrad{\varphi}^2\,\d\eta\biggr)^{1/2}+\biggl(\int_E\weakgrad{\sqrt{g'}}^2\,\d\mm\biggr)^{1/2}
\biggl(\int_E\weakgrad{\varphi}^2\,\d\eta'\biggr)^{1/2}.\nonumber
\end{eqnarray}
\end{lemma}
\begin{proof} By Lemma~\ref{lem:WeakConvK} we can assume, by a simple approximation argument, that $\varphi$ has bounded support.
Under this assumption the quantity to be estimated reduces, thanks to \eqref{eq:transfer1} and \eqref{eq:localityGamma}, to
$$
\left|\int_X\Gbil{\varphi}{g}-\Gbil{\varphi}{g'}\,\d\mm\right|=
\left|\int_E\Gbil{\varphi}{g}-\Gbil{\varphi}{g'}\,\d\mm\right|\leq
\int_E\bigl(\weakgrad{g}\weakgrad{\varphi}+\weakgrad{g'}\weakgrad{\varphi}\bigr)\,\d\mm
$$
and, after dividing and multiplying by $\sqrt{g}$ and $\sqrt{g'}$, we can use H\"older's inequality to provide the result.
\end{proof}

\section{Equivalence of the different formulations of $RCD(K,\infty)$}\label{sec:lastq}

In this section we prove the following result, extending Theorem~\ref{thm:main} to 
$\sigma$-finite metric measure spaces. 

\begin{theorem}\label{thm:main1}
 Let $(X,\sfd,\mm)$ be a metric measure space with $(X,\sfd)$ complete, separable, $\mm$ finite on bounded sets
 and with $\supp\mm=X$. Then the following properties are equivalent.
 \begin{enumerate}
  \item[(i)] $(X,\sfd,\mm)$ is a $CD(K,\infty)$ space and the semigroup $\heatw_t$ on $\ProbabilitiesTwo{X}$ is additive.
  \item[(ii)] $(X,\sfd,\mm)$ is a $CD(K,\infty)$ space and $\C$ is a quadratic form on $L^2(X,\mm)$.
  \item[(iii)] $(X,\sfd,\mm)$ is a length space, \eqref{eq:growthcond} holds and any $\mu \in \ProbabilitiesTwo{X}$ is the starting point of an $EVI_K$
                gradient flow of $\entv$.
 \end{enumerate}
Any metric measure space $(X,\sfd,\mm)$ satisfying these assumptions and one of the equivalent properties (i), (ii), (iii) will be
called ($\sigma$-finite) $RCD(K,\infty)$ space.
\end{theorem}

Here $\heatw_t$ is the
$W_2$-gradient flow of $\entv$, according to Definition~\ref{def:dissKconv} (which is known to exist and to be unique for any given initial datum in $D({\rm Ent}_\mm)$, see \cite{Gigli10} and \cite{Ambrosio-Gigli-Savare11}), while 
${\sf h}_t$ stands for the gradient flow of $\C$ in $L^2(X,\mm)$ (or, equivalently, the $EVI_0$ gradient flow).
 
Note that the implications (i) to (ii) and (iii) to (i) were already proved by the first two
authors with Savar\'e in \cite{Ambrosio-Gigli-Savare11b}, because the same proof works in the $\sigma$-finite
setting. The key
implication from (ii) (or (i)) to (iii) is given by the derivative of quadratic optimal transport distance along the heat flow
and of the entropy
along a geodesic, estimated in the next two subsections. Consequently we shall always assume in this section
that $\C$ is quadratic.

We denote by $\Delta$ the infinitesimal generator of the linear semigroup ${\sf h}_t$, so that 
$$
\frac{\d}{\d t}{\sf h}_tf=\Delta {\sf h}_t\qquad\text{for a.e. $t>0$.}
$$
Also, since $\C$ is quadratic, $\Delta$ is related to the bilinear form $\mathcal E$ in \eqref{eq:numeriamoanchequesta}
by
\begin{equation}\label{eq:intbypartsoo}
\int_X g\Delta f\,\d\mm=\mathcal E(f,g)\qquad\forall g\in\calW\cap L^2(X,\mm),\,\,f\in D(\Delta).
\end{equation}

One of the main result of the work of the first two authors with Savar\'e \cite{Ambrosio-Gigli-Savare11} has been
the following identification theorem in $CD(K,\infty)$, see  (8.5), Theorem~8.5 and
Theorem~9.3(iii) therein.

\begin{theorem}[The heat flow as gradient flow]\label{thm:heatgf}
Let $(X,\sfd,\mm)$ be a $CD(K,\infty)$ space and let $f\in
L^2(X,\mm)$ be such that $\mu=f\mm\in\probt{X}$. 
Then $\heatw_t\mu={\sf h}_tf\mm$ for all $t\geq 0$, $t\mapsto\entv(\heatw_t\mu)$ is
locally absolutely continuous in $[0,\infty)$, and
\begin{equation}\label{eq:edissrateflow}
-\frac{\d}{\d
t}\entv(\heatw_t\mu)=|\dot{\heatw_t\mu}|^2=\int_{\{{\sf h}_tf>0\}}\frac{\weakgrad
{{\sf h}_tf}^2}{{\sf h}_tf}\,\d\mm\qquad\text{for a.e.~$t>0$.}
\end{equation}
\end{theorem}

In other words, one can unambiguously define the heat flow on a
$CD(K,\infty)$ space either as the gradient flow of $\C$
in $L^2(X,\mm)$ or as the $W_2$-gradient flow of $\entv$.

\subsection{Derivative of $W_2^2(\cdot,\sigma)$ along the heat flow }

Notice that this result, whose proof is achieved by a duality argument, requires no curvature assumption.
We need only to assume that $\C$ is quadratic and that $\mm$ satisfies the growth condition \eqref{eq:growthcond}. 

\begin{theorem}
\label{thm:derw2} Let 
$\mu=f\mm\in D(\entv)$ and define $\mu_t:=({\sf h}_tf)\mm=f_t\mm$.
Let $\sigma\in\probt X$ with bounded support. Then, for a.e. $t>0$ the following property holds:
for any Kantorovich potential $\varphi_t$ relative to $(\mu_t,\sigma)$ whose
slope has linear growth, one has
\begin{equation}
\label{eq:derw2} \frac{\d}{\d t}\frac12W_2^2(\mu_t,\sigma)=
-\mathcal E_{\mu_t}(\varphi_t,\log f_t).
\end{equation}
\end{theorem}
\begin{proof} By the energy dissipation estimate \eqref{eq:edissrateflow},
we have $\int_0^\infty\C(\sqrt{f_t})\,\dt<\infty$. Furthermore, the maximum principle
proved in Theorem~4.20 of \cite{Ambrosio-Gigli-Savare11} shows that
$f_t\leq \|f\|_\infty$ $\mm$-a.e. in $X$ for all
$t\geq 0$. Also, by Proposition~\ref{prop:goodKant} the potential $\varphi_t$ belongs to
$L^1(X,\nu)$ for all $\nu\in\probt{X}$ and its slope has linear growth. Furthermore, the
$L^1$ estimate is uniform in $t$ and in bounded subsets of $\probt
X$ and the estimate on the slope depends on $\sigma$ only.

Thanks to \eqref{eq:edissrateflow}, the map $t\mapsto f_t\mm$ is a locally absolutely continuous curve
in $\probt X$, hence the derivative on the left hand side of
\eqref{eq:derw2} exists for a.e.~$t>0$. Also, the derivative of
$t\mapsto f_t$ exists in $L^2(X,\mm)$ and coincides with $\Delta f_t$ for a.e.~$t>0$.
Fix $t_0>0$ where both properties hold, which is also
a Lebesgue point for $\C(\sqrt{f_t})$.

We now claim that 
\begin{equation}\label{eq:tyuh}
\lim_{h\downarrow 0}\int_X\psi\frac{
f_{t_0}-f_{t_0-h}}{h}\,\d\mm=-{\mathcal
E}_{\mu_{t_0}}(\psi,\log f_{t_0})
\end{equation} 
for all locally Lipschitz functions $\psi$ whose gradient has linear growth. 
The proof of \eqref{eq:tyuh} is easy if we assume, in addition, that
$\psi$ has bounded support. Indeed,
$h^{-1}(f_{t_0+h}-f_{t_0})\to\Delta f_{t_0}$ as $h\to 0$ in
$L^2(X,\mm)$, so that \eqref{eq:transfer1} and \eqref{eq:intbypartsoo} give
$$
\lim_{h\to 0}\int_X\psi\frac{
f_{t_0+h}-f_{t_0}}{h}\,\d\mm= \int_X\psi\Delta f_{t_0}\,\d\mm=
-{\mathcal E}(\psi,f_{t_0})=-{\mathcal
E}_{\mu_{t_0}}(\psi,\log f_{t_0}).
$$

For the general case, let $\chi_N:X\to [0,1]$ be satisfying ${\rm
Lip}(\chi_N)\leq 1$, $\chi_N\equiv 1$ on $B_N(x_0)$ and
$\chi_N\equiv 0$ on $X\setminus B_{2N}(x_0)$ and define
$\psi^N:=\psi\chi_N$. Applying Lemma~\ref{le:luigi} below with
$\varphi_N:=\psi-\psi^N$ we get
\[
\sup_{|h|<t_0/2}\left|\int
\varphi_N\frac{\rho_{t_0+h}-\rho_{t_0}}{h}\,\d\mm\right|^2\\
\leq \sup_{|h|<t_0/2}\frac8h\int\limits_{t_0-|h|}^{t_0+|h|}
 \C(\sqrt{f_s})\int_X\weakgrad{\varphi_N}^2\,\d\mu_s\,\d s.
\]
Hence (by our choice of $t_0$ and the
$2$-uniform integrability of $\mu_s$)
\[
\limsup_{N\to\infty}\sup_{|h|<t_0/2}\left|\int_X\varphi_N\frac{
f_{t_0+h}-f_{t_0}}{h}\,\d\mm\right|=0,
\]
which, taking into account that ${\mathcal
E}_{\mu_{t_0}}(\psi^N,\log f_{t_0})\to {\mathcal
E}_{\mu_{t_0}}(\psi,\log f_{t_0})$ thanks to Lemma~\ref{lem:WeakConvK}, 
implies \eqref{eq:tyuh}.

Now, notice that since
$\varphi_{t_0}$ is a Kantorovich potential for $(\mu_{t_0},\sigma)$
one has
\[
\begin{split}
\frac 12 W_2^2(\mu_{t_0},\sigma)&=\int_X\varphi_{t_0}\,\d\mu_{t_0}+\int\varphi_{t_0}^c\,\d\sigma\\
\frac 12
W_2^2(\mu_{t_0-h},\sigma)&\geq\int_X\varphi_{t_0}\,\d\mu_{t_0-h}+
\int\varphi_{t_0}^c\,\d\sigma\qquad\text{for all $h$ such that
$t_0-h> 0$.}
\end{split}
\]
Taking the difference between the first identity and the second
inequality and using the claim with $\psi=\varphi_{t_0}$ we get
\[
\frac 12 W_2^2(\mu_{t_0+h},\sigma)-\frac 12 W_2^2(\mu_{t_0},\sigma)\geq -h\mathcal E_{\mu_{t_0}}(\log f_{t_0},\varphi_{t_0})+o(h).
\]
Since $t\mapsto W_2^2(\mu_t,\sigma)$ is differentiable at $t=t_0$ we conclude.
\end{proof}

\begin{lemma}\label{le:luigi} Let $\mu_s=f_s\mm$ be as in the previous theorem and
let $\varphi:X\to\R$ be locally Lipschitz, with $|D \varphi|$ having linear
growth. Then, for $[s,t]\subset (0,\infty)$ one has
\begin{equation}
\label{eq:trucco} \left|\int
\varphi\frac{f_t-f_s}{t-s}\,\d\mm\right|^2\leq \frac
8{t-s}\int_s^t \C(\sqrt{f_r})\biggl(\int\weakgrad
\varphi^2\,\d\mu_r\biggr)\,\d r.
\end{equation}
\end{lemma}
\begin{proof}
Assume first that $\varphi\in L^2(X,\mm)$. Then integrating by parts we
get
\[
\left|\int \varphi\Delta f_r\,\d\mm\right|^2\leq\left(\int\weakgrad
\varphi\,\weakgrad{f_r}\,\d\mm\right)^2\leq\int\weakgrad
\varphi^2\,\d\mu_r\,\int\frac{\weakgrad{f_r}^2}{f_r}\,\d\mm,
\]
for all $r>0$, and the thesis follows by integration in $(s,t)$. For
the general case, we approximate $\varphi$ by $\varphi\chi_N$, with $\chi_N$
chosen as in the proof of the previous theorem.
\end{proof}

\subsection{Derivative of the entropy along $\entv$-convex $L^\infty$-bounded geodesics}

The goal of this subsection is to prove the following theorem, where both the curvature condition and the fact
that $\C$ is quadratic play a role.

\begin{theorem}[Entropy inequality]\label{Thm:DerEntr}
Assume that $(X,\sfd,\mm)$ is a $CD(K,\infty)$ space.
Let  $\eta=f\mm,\,\sigma=g\mm\in\probt X$ with $g$ uniformly bounded and having compact support, $f$ uniformly bounded with
$\C(\sqrt{f})<\infty$.  Then there exists a Kantorovich potential $\varphi$ from $\eta$ to $\sigma$ such that $|D\varphi|$ has linear growth
and
\begin{equation}\label{eq:Step1}
\entv(\sigma)-\entv(\eta)- \frac{K}{2} W_2^2(\eta,\sigma) \geq -\mathcal E_{\eta}(\varphi,\log f).
\end{equation}
\end{theorem}

The proof of Theorem~\ref{Thm:DerEntr}, carried by approximation, is presented at the end of the subsection; 
the first crucial step is the following proposition, whose proof relies on
Proposition~\ref{prop:goodKant} and Lemma \ref{lem4.4revised}.

\begin{proposition}\label{Prop1}
Under the assumptions of Theorem~\ref{Thm:DerEntr}, for $\delta>0$ call 
\begin{equation}\label{eq:chietad}
f_{\delta,n}=c_{\delta,n}[(\chi_n^2) \eta \vee \delta\rme^{-2c\Wgh^2}],
\end{equation}
where $c$ is strictly larger than the constant ${\sf c}$ in \eqref{eq:growthcond}, 
$c_{\delta,n}$ is the normalizing constant such that $f_{\delta,n} \mm$ is a probability density, 
$\chi_n$ is a $1$-Lipschitz cut-off function equal to $1$ on $B_n(x_0)$ and null outside $B_{2n}(x_0)$.\\
Then there exists a Kantorovich potential $\varphi_{\delta,n}$ from $\eta_{\delta,n}:=f_{\delta,n} \mm$ to $\sigma$ 
satisfying the growth conditions
\begin{equation}\label{growthPhidn}
|\varphi_{\delta,n}(x)|\leq C(\sigma)(1+\sfd^2(x,x_0)),\quad\quad\quad |D\varphi_{\delta,n}|(x)\leq C(\sigma)(1+\sfd(x,x_0)),
\end{equation}
such that
\begin{equation}\label{eq:Prop1}
\entv(\sigma)-\entv(\eta_{\delta,n})- \frac{K}{2} W_2^2(\eta_{\delta,n},\sigma) \geq 
-\mathcal E_{\eta_{\delta,n}}(\varphi_{\delta,n},\log f_{\delta,n}).
\end{equation}
\end{proposition}
\begin{proof} First of all we are under the assumptions of Theorem~\ref{thm:goodgeodesics}, so let $\ppi\in \gopt(\eta_{\delta,n},\sigma)$ and let 
$(\e_t)_\sharp\ppi=\mu_t=f_t\mm$, $t\in[0,1]$, be the associated good geodesic from $\eta_{\delta,n}$ to $\sigma$ with a
uniform  $L^\infty$ bound on the density for $t\in (0,t_0)$ and the $K$-convexity of the entropy.  Let also $\varphi$ be the Kantorovich potential, 
given by Proposition~\ref{prop:goodKant}, with quadratic growth and whose slope has linear growth.

Let us now check that $f_{\delta,n}$ satisfies the assumptions of Lemma~\ref{lem4.4revised}. Indeed, $|D \log f_{\delta,n}| \leq C(1+\sfd(x,x_0))$ 
whenever $\sfd(x,x_0)>2n$, because in this set $f_{\delta_n}$ coincides with $c_{\delta,n}\delta\rme^{-2c\Wgh^2}$; in addition, the locality of weak
gradients and the partition  $X=\{\chi^2_n \eta>\delta\rme^{-2c\Wgh^2} \} \cup \{\chi^2_n \eta\leq \delta\rme^{-2c\Wgh^2}\}$ ensure
that $\weakgrad{\log f_{\delta,n}}\in L^2(X,\eta_{\delta,n})$ because  the finiteness of $\C(\sqrt{f})$ ensures that $\weakgrad{\log f}\in L^2(X,\eta)$.

Observe that the convexity of $z\mapsto z \log z$ gives
\begin{equation}\label{eq:logpi}
\frac{\entv(\mu_t)-\entv(\eta_{\delta,n})}{t} \geq \int_X \log f_{\delta,n} \frac{f_t-f_{\delta,n}}{t} \,\d \mm = 
\int \frac{\log(f_{\delta,n}\circ\e_t)-\log(f_{\delta,n} \circ\e_0)}{t}\, \d \ppi.
\end{equation}
Define the functions $F_t,\,G_t: AC^2([0,1];X)\to \R$ as
\begin{equation}
F_t(\gamma):=\frac{\log(f_{\delta,n} \circ\e_0)-\log(f_{\delta,n} \circ \e_t)}{\sfd(\gamma_0, \gamma_t)} \label{def:Ft},\qquad
G_t(\gamma):=\frac{\varphi\circ\e_0-\varphi\circ\e_t}{\sfd(\gamma_0, \gamma_t)}.
\end{equation}
Multiplying and dividing the right hand side of \eqref{eq:logpi} by $\sfd(\gamma_0,\gamma_t)$ we obtain
\begin{equation}\label{QLastEq}
\liminf_{t \down 0} \frac{\entv(\mu_t)-\entv(\eta_{\delta,n} \mm)}{t}\geq 
- \limsup_{t\down 0} \int F_t(\gamma) \frac{\sfd(\gamma_0,\gamma_t)}{t} \d \ppi(\gamma).%= -\limsup_{t \down 0} \int F_t G_t \d \ppi,
\end{equation} 
Now we claim that
\begin{equation}\label{QLastEqbis}
- \limsup_{t\down 0} \int F_t(\gamma) \frac{\sfd(\gamma_0,\gamma_t)}{t} \d \ppi(\gamma)= -\limsup_{t \down 0} \int F_t G_t \d \ppi.
\end{equation} 
The proof of \eqref{QLastEqbis} follows at once by
\begin{equation}\label{eq:chepalle}
\lim_{t\down 0 } \int\left| G_t(\gamma)-\frac{\sfd(\gamma_0,\gamma_t)}{t}\right |^2\,\d\ppi=0
\quad \text{and} \quad \sup_{t\leq t_0} \int |F_t|^2\,\d\ppi<\infty.
\end{equation}
The first fact in \eqref{eq:chepalle} is ensured by (2) of Lemma~\ref{lem4.4revised}, as well as the
identity
\begin{equation}\label{eq:logope}
\int \weakgrad{\varphi}^2\circ\e_0\,\d \ppi=\lim_{t\downarrow 0}\int |G_t|^2 \d\ppi.
\end{equation}
The second fact in \eqref{eq:chepalle} is ensured by (1) of the same lemma applied 
to $h=\log f_{\delta,n}$. Combining \eqref{QLastEq} and \eqref{QLastEqbis} we get
\begin{equation}\label{QLastEqter}
\liminf_{t \down 0} \frac{\entv(\mu_t)-\entv(\eta_{\delta,n} \mm)}{t}\geq 
-\limsup_{t \down 0} \int F_t G_t \d \ppi.
\end{equation} 
Now, applying Lemma~\ref{lem4.4revised} to $h=\varphi+\epsilon \log f_{\delta,n}$ gives that
\begin{equation}
\int \weakgrad{(\varphi+\epsilon \log f_{\delta,n})}^2\circ\e_0 \d \ppi \geq
\limsup_{t\down 0} \int |G_t(\gamma)+\epsilon F_t(\gamma)|^2 \d \ppi(\gamma) \label{eq:HorVert}.
\end{equation}
Subtracting to \eqref{eq:HorVert} the equality \eqref{eq:logope} and dividing by $\epsilon$ gives
\begin{equation}\label{lastEq}
\limsup_{t\down 0} \int G_t F_t \,\d \ppi \leq \liminf_{\epsilon \down 0} 
\int_X \frac{\weakgrad{(\varphi+\epsilon \log f_{\delta,n})}^2-\weakgrad{\varphi}^2}{2\epsilon} f_{\delta,n} \,\d \mm
= \mathcal E_{\eta_{\delta,n}} (\log f_{\delta,n},\varphi),
\end{equation}
where we used again the uniform bound on the $L^2$ norm of $F_t$.
Combining  \eqref{QLastEqter} and \eqref{lastEq} we obtain
\begin{equation}\label{eq:Last1}
\liminf_{t \down 0} \frac{\entv(\mu_t)-\entv(\eta_{\delta,n})}{t}\geq - \mathcal E_{\eta_{\delta,n}} (\log f_{\delta,n},\varphi).
\end{equation}
The conclusion follows by \eqref{eq:Last1} recalling that, by construction, the entropy is $K$-convex along the geodesic $(\mu_t)_{t\in [0,1]}$, see \eqref{eq:CDdef}. 
\end{proof}

\noindent
{\bf Proof of Theorem~\ref{Thm:DerEntr}.} In this proof we denote for brevity $a\vee b=\max\{a,b\}$. For every $\delta\in (0,1)$ define the density
\begin{equation}\label{def:etadelta}
\tilde{f}_{\delta}:=f \vee (\delta \rme^{-2c\Wgh^2}) \quad\text{and } f_{\delta}:=c_{\delta} \tilde{f}_{\delta} \text{ with $c_\delta\uparrow 1$
as $\delta\downarrow 0$} 
\end{equation}
(here $c>0$ is the constant in \eqref{eq:growthcond}),
so that $\tilde{f}_\delta \geq f $ and $c_{\delta}$ are the normalizing constants. We need a further regularization of $f_\delta$; to this aim,
let $\chi_n$ be standard cut-off functions, namely 
$0 \leq \chi_n \leq 1$, $\Lip (\chi_n) \leq 1$, $\chi_n\equiv 1$ on  $B_n(x_0)$ and $\chi_n\equiv 0$ on $B^c_{2n}(x_0)$. Then,
for every $n>1,\,\delta>0$ we define the densities
\begin{equation}\label{def:etadeltan}
\tilde{f}_{\delta,n}:= (\chi^2_n f) \vee (\delta \rme^{-2c\Wgh^2}) \quad\text{and } f_{\delta,n}:=c_{\delta,n} \tilde{f}_{\delta,n} \text{ with } \; c_{\delta,n}\downarrow c_\delta \text{ as } n\to\infty, 
\end{equation}
so that $\tilde{f}_{\delta,n} \leq \tilde{f}_\delta $ and $c_{\delta,n}$ are the normalizing constants.
Of course $f_{\delta,n}$ is uniformly bounded and $\eta_{\delta,n}:=f_{\delta,n} \mm \in \ProbabilitiesTwo X$, moreover $\C(\sqrt{f_{\delta,n}})$ is
finite. Indeed by the chain rule and the locality of the weak gradients we have that
\begin{eqnarray*}
\weakgrad{\sqrt{f_{\delta,n}}} & =&  \sqrt{c_{\delta,n}}|D(\chi_n \sqrt{f})|_w\\
  &\leq& \sqrt{c_{\delta,n}} \left( \chi_n \weakgrad{\sqrt{f}}+ \sqrt{f}\weakgrad{\chi_n} \right) \quad\text{if $\chi^2_nf \geq \delta \rme^{-2c\Wgh^2}$}\\
\weakgrad{\sqrt{f_{\delta,n}}} & =& \sqrt{\delta\, c_{\delta,n} } \weakgrad{\rme^{-2c \Wgh^2}} \\
& \leq& 4c \sqrt{\delta \,c_{\delta,n}}\, \sfd(\cdot,x_0) \,e^{-2c \Wgh^2} \quad\text{otherwise}.
\end{eqnarray*}
Since by assumption $\C(\sqrt{f}) < \infty$, it follows not only that $\weakgrad{\sqrt{f_{\delta,n}}}^2$ are uniformly
bounded in $L^1(X,\mm)$, but also that they are equi-integrable:
\begin{equation}\label{eq:equiBoundF}
\sup_{\delta\in (0,1),\, n \in \N } \C(\sqrt{f_{\delta,n}})< \infty\quad\text{and}\quad
E_j\downarrow\emptyset\,\,\Rightarrow\sup_{\delta\in (0,1),\,n\in\N}\int_{E_j}\weakgrad{\sqrt{f_{\delta,n}}}^2\,\d\mm\to 0.
\end{equation} 
Observe that $(\eta_{\delta,n},\sigma)$ has the structure described in Proposition~\ref{Prop1}, so there exists a Kantorovich potential 
$\varphi_{\delta,n}$ from $\eta_{\delta,n}$ to $\sigma$ satisfying the growth conditions \eqref{growthPhidn}
and such that the entropy inequality holds:
\begin{equation}\label{eq:EntEstAppr}
\entv(\sigma)-\entv(\eta_{\delta,n})- \frac{K}{2} W_2^2(\eta_{\delta,n},\sigma) 
\geq -\mathcal E_{\eta_{\delta,n}}(\varphi_{\delta,n},\log f_{\delta,n}).
\end{equation}

\noindent
{\bf Passage to the limit as $n\to\infty$.} Consider the transportation problem from $\eta_\delta:=f_\delta \mm$ to $\sigma$. 
We claim the existence of a Kantorovich potential $\varphi_\delta$ such that 
\begin{equation}\label{eq:EntEstdelta}
\entv(\sigma)-\entv(\eta_{\delta})- \frac{K}{2} W_2^2(\eta_{\delta},\sigma) \geq -\mathcal E_{\eta_{\delta}}(\varphi_\delta,\log f_{\delta}).
\end{equation}
We would like to pass to the limit as $n\to\infty$ in \eqref{eq:EntEstAppr}. Let us start by considering the left hand side: applying 
Lemma~\ref{lem:ConvEnt}  to $\tilde{\eta}_{\delta,n}\up \tilde{\eta}_{\delta}$ $\mm$-a.e,  and recalling that $c_{\delta,n}\down c_{\delta}$ as $n\to \infty$, 
we get
\begin{equation}\label{eq:ConvEnteta}
\entv (\eta_{\delta,n}) \to  \entv (\eta_\delta) \quad \text{as } n\to \infty.
\end{equation}
It is easy to check that $\eta_{\delta,n} $ weakly converge to $\eta_{\delta}$ and have uniformly integrable 2-moments, so by 
\cite[Proposition~7.1.5]{Ambrosio-Gigli-Savare08} we have 
\begin{equation}\label{eq:ConvW}
\lim_{n\to\infty} W_2^2(\eta_{\delta,n} ,\sigma) = W_2^2(\eta_\delta,\sigma).
\end{equation}
Now let us show the convergence of the right hand side of \eqref{eq:EntEstAppr}. To simplify the problem we prove first that
\begin{equation}\label{eq:etadnetad}
\lim_{n\to\infty} \bigl| \mathcal E_{\eta_{\delta,n}}(\varphi_{\delta,n},\log f_{\delta,n})-\frac{c_{\delta,n}}{c_{\delta}}
\mathcal E_{\eta_{\delta}}(\varphi_{\delta,n},\log f_{\delta}) \bigr|=0.
\end{equation}
Notice that, calling $A_\delta:=\{x\in X:\ f(x)\geq \delta\rme^{-2c\Wgh^2(x)} \}$ we have $f_{\delta,n}=\frac{c_{\delta,n}}{c_\delta}f_\delta$
on the complement $(A_{\delta}\cap B_n(x_0))\cup A_\delta^c$
of $A_\delta\setminus B_n(x_0)$.  Since $A_\delta\setminus B_n(x_0)\downarrow\emptyset$
we can use \eqref{eq:essential} of Lemma~\ref{lem:essential} to obtain \eqref{eq:etadnetad}, taking
\eqref{eq:equiBoundF} into account.

{F}rom \eqref{eq:etadnetad}, and taking into account that $c_{\delta,n}\to c_\delta$ as $n\to\infty$, 
 in order to prove the convergence of the right hand side of \eqref{eq:EntEstAppr}, it is enough to show the existence of
a Kantorovich potential $\varphi_\delta$ for $(\eta_\delta,\sigma)$ such that
\begin{equation}\label{reduct:convEdn}
\mathcal E_{\eta_\delta} (\varphi_{\delta,n}, \log f_{\delta}) \to \mathcal E_{\eta_\delta} (\varphi_{\delta}, \log f_{\delta})\quad \text{as } n\to \infty. 
\end{equation}
Now we use in a crucial way Lemma~\ref{lem:GammaConvKant}, which ensures the existence of a Kantorovich potential $\varphi_\delta$ 
for $(\eta_\delta,\sigma)$ and of a subsequence $n(k)$ such that $\varphi_{\delta,n(k)} \to \varphi_{\delta}$ pointwise in $X$. Recalling
that $|\varphi_{\delta,n}|\leq C(1+\Wgh^2)$ and that $\int |D \varphi_{\delta,n}|_w^2 \,\d\eta_\delta$ is uniformly bounded, 
we are in position to apply Lemma~\ref{lem:WeakConvK} and to conclude that \eqref{reduct:convEdn} holds. 
Therefore we proved the convergence of all terms in \eqref{eq:EntEstAppr}, so that \eqref{eq:EntEstdelta} holds.

\noindent
{\bf Passage to the limit as $\delta\downarrow 0$.} The inequality \eqref{eq:EntEstdelta} passes to the limit as $\delta\down 0$: more precisely,
 we claim the existence of a Kantorovich potential $\varphi$ from $f\mm$ to $\sigma$ such that 
\begin{equation}\label{eq:EntEstdel}
\entv(\sigma)-\entv(\eta)- \frac{K}{2} W_2^2(\eta,\sigma)\geq -\mathcal E_\eta(\varphi,\log f).
\end{equation}
As in the passage to the limit as $n\to\infty$,  Lemma~\ref{lem:ConvEnt} easily implies that
$\entv(\eta_{\delta})\to\entv(\eta)$,
moreover it is easy to check that $\eta_\delta$ weakly converge to $\eta$ and have uniformly integrable 2-moments, so 
\cite[Proposition~7.1.5]{Ambrosio-Gigli-Savare08} gives
$W_2(\eta_\delta,\sigma)\to W_2(\eta,\sigma)$.
In order to show the convergence of the right hand side of \eqref{eq:EntEstdel} we first prove that
\begin{equation}\label{eq:EdE}
\lim_{\delta\down 0}|\mathcal E_{\eta_\delta}(\varphi_\delta, \log f_\delta)-c_\delta\mathcal E_{\eta} (\varphi_\delta, \log f)|=0.
\end{equation}
First of all notice that, after calling $A_\delta:=\{x\in X: f(x)\geq \delta\rme^{-2c \Wgh^2(x)} \}$, we have $f_\delta=c_\delta f$
on $A_\delta$. Since $X\setminus A_\delta\downarrow \{f=0\}$ as $\delta\downarrow 0$ and $\weakgrad{f}=0$
$\mm$-a.e. on $\{f=0\}$,  we can use \eqref{eq:essential} 
of Lemma~\ref{lem:essential} to show \eqref{eq:EdE}, taking \eqref{eq:equiBoundF} into account.

Now that \eqref{eq:EdE} is proved, taking into account that $c_\delta\to 1$ as $\delta\downarrow 0$, 
it is enough to prove the existence of a Kantorovich potential $\varphi$ from $\eta$ to $\sigma$ such that
\begin{equation}\label{eq:QF}
\lim_{i\to\infty}  \mathcal E_\eta (\varphi_{\delta_i}, \log f)=\mathcal E_\eta (\varphi, \log f).
\end{equation}
for some sequence $\delta_i\downarrow 0$.
Recall that $\varphi_\delta$ were constructed using Lemma~\ref{lem:GammaConvKant}, so they still satisfy the growth condition \eqref{growthPhidn}; 
applying again Lemma~\ref{lem:GammaConvKant} we get the existence of a Kantorovich potential $\varphi$ from $\eta$ to $\sigma$ and
$\delta_i\downarrow 0$ such that 
$\varphi_{\delta_i}\to \varphi$ pointwise in $X$ as $i \to \infty$. Moreover, by  \eqref{eq:itforza}
and $f\leq c^{-1}_\delta f_\delta\leq 2f_\delta$ 
for $\delta$ small enough, we have 
$$\int_X \weakgrad{\varphi_{\delta_i}}^2\; f\,\d \mm \leq 2\int_X \weakgrad{\varphi_{\delta_i}}^2 f_{\delta_i}\, \d \mm \leq 2 
W^2_2(\eta_{\delta_i},\sigma),$$ 
for $i$ large enough. Hence we can apply Lemma~\ref{lem:WeakConvK} and conclude that \eqref{eq:QF} holds.
Therefore \eqref{eq:EntEstdel} is proved and the proof of Theorem~\ref{Thm:DerEntr} is then complete.
\hfill$\Box$

\subsection{Proof of Theorem~\ref{thm:main1}.} 

The implications from (i) to (ii) and from (iii) to (i) can be proven exactly
as in Theorem~5.1 of \cite{Ambrosio-Gigli-Savare11b} (as these proofs need no finiteness assumption on $\mm$), 
so let us focus on the implication from (ii) to (iii). Note that Sturm has proven in \cite{Sturm06I} (see Remark~4.6(iii) therein)
that $\supp\mm$ is a length space for all $CD(K,\infty)$ spaces $(X,\sfd,\mm)$ (his proof, based on an approximate
midpoint construction, does not use the local compactness).

It remains to show that the $EVI_K$-condition holds assuming
the $CD(K,\infty)$ condition and the fact that $\C$ is quadratic.
By the contractivity properties of $EVI_K$-gradient flows stated in Proposition~\ref{prop:evipropr} it is sufficient to
show that $\mu_t:=({\sf h}_tf)\mm$ is an $EVI_K$ gradient flow for $\entv$
for any initial measure $f\mm\in\probt{X}$ whose density $f$ is bounded and
satisfies $\C(\sqrt{f})<\infty$. By the maximum principle
proven in \cite{Ambrosio-Gigli-Savare11} (see Theorem~4.20 therein) one has
${\sf h}_tf\leq\|f\|_{L^\infty(X,\mm)}$ $\mm$-a.e. in $X$ for all $t\geq 0$, furthermore
$\{\mu_t:\ t\in [0,T]\}$ is a bounded subset of $\probt{X}$ for all $T>0$
and \eqref{eq:edissrateflow} gives
\begin{equation}
\int_0^\infty \C(\sqrt{{\sf h}_t f})\,\dt<\infty.
\end{equation} 
By a simple density argument on the class of ``test'' measures $\sigma$ in \eqref{eq:EVI}
(see for instance \cite[Proposition~2.20]{Ambrosio-Gigli-Savare11b}), we can restrict ourselves to measures
$\sigma$ of the form $g\mm$ with $g\in L^\infty(X,\mm)$ and $\supp\sigma$ compact.

By \eqref{eq:derw2} of Theorem~\ref{thm:derw2} we get that for a.e. $t>0$, for any choice
of a Kantorovich potential $\varphi_t$ from $\mu_t$ to $\sigma$ whose slope has linear growth,
one has
\begin{equation}
\label{eq:derw22} \frac{\d}{\d t}\frac12W_2^2(\mu_t,\sigma)=
-\mathcal E_{\mu_t}(\varphi_t,\log {\sf h}_tf).
\end{equation}
Therefore, to conclude that \eqref{eq:EVI} holds, it suffices to show for a.e. $t>0$ the existence of
a Kantorovich potential $\varphi_t$ from $\mu_t$ to $\sigma$ whose slope has linear growth and
satisfies
\begin{equation}
-\mathcal E_{\mu_t}(\varphi_t,\log {\sf h}_t f)\leq \entv(\sigma)-\entv(\mu_t)- \frac{K}{2} W_2^2(\mu_t,\sigma). 
\end{equation}
This is precisely the statement of Theorem~\ref{Thm:DerEntr} (with $\eta=\mu_t$) 
and this concludes the proof. \hfill$\Box$

\section{Properties of $RCD(K,\infty)$ spaces}\label{sec:last}

In this section we state without proof some properties of $RCD(K,\infty)$ spaces whose proofs, given by the first two authors
and Savar\'e in \cite{Ambrosio-Gigli-Savare11b}. Their proofs do not rely on the finiteness assumption of $\mm$. Refer to \cite{Ambrosio-Gigli-Savare11b}
for details of proofs and a more complete discussion. 

\subsection{The heat semigroup and its regularizing properties}

In this section we describe more in detail the properties of the
$L^2$-semigroup ${\sf h}_t$ in a $RCD(K,\infty)$ space
and the additional information that one can
obtain from the identification with
$W_2$-semigroup $\heatw_t$. 
By the definition of $RCD(K,\infty)$ spaces, we know that for any $x\in X$ there
exists a unique $EVI_K$ gradient flow $\heatw_t(\delta_x)$ of
$\entv$ starting from $\delta_x$, related to ${\sf h}_t$ by
\begin{equation}\label{eq:heatlheatw}
({\sf h}_t f)\mm=\int f(x)\,\heatw_t(\delta_x)\,\d\mm(x)
\qquad\forall f\in L^2(X,\mm).
\end{equation}
Since $\entv(\ke xt)<\infty$ for any $t>0$, one has $\ke
xt\ll\mm$, so that $\ke xt$ has a density, that we shall denote by
$\ked xt$. The functions $\ked xt (y)$ are the so-called transition
probabilities of the semigroup. By standard measurable selection
arguments we can choose versions of these densities in such a way
that the map $(x,y)\mapsto\ked xt(y)$ is $\mm\times\mm$-measurable
for all $t>0$.

In the next theorem we prove additional properties of the flows. The
information on both benefits of the identification theorem: for
instance the symmetry property of transition probabilities is not at
all obvious when looking at $\heatw_t$ only from the optimal
transport point of view, and heavily relies on
\eqref{eq:heatlheatw}. On the other hand, the regularizing
properties of ${\sf h}_t$ are deduced by duality by those of
$\heatw_t$, using in particular the contractivity estimate
\begin{equation}
\label{eq:contrw2} W_2(\heatw_t(\mu),\heatw_t(\nu))\leq
e^{-Kt}W_2(\mu,\nu)\qquad t\geq0,\ \mu,\,\nu\in\probt {X,\mm}
\end{equation}
and the regularization estimates for the Entropy and its slope
\begin{equation}
  \label{eq:24}
  \mathrm I_K(t)\ent{\heatw_t(\mu)}+\frac {(\mathrm I_K(t))^2}2|D^- \entv|^2(\heatw_t(\mu))
  \le \frac 12 W_2^2(\mu,\mm)
\end{equation}
which are typical of $EVI_K$-solutions, with $\mathrm I_K(t):=\int_0^t \rme^{K r}\,\d r$.
Notice also that \eqref{eq:contrw2} yields $
W_1(\heatw_t(\delta_x),\heatw_t(\delta_y))\le\rme^{-Kt}\sfd(x,y)$
for all $x,\,y\in X$ and $t\geq 0$. This implies that $RCD(K,\infty)$
spaces have Ricci curvature bounded from below by $K$
according to the $W_1$-contractivity property taken as definition in Ollivier \cite{Ollivier09} and Joulin \cite{Joulin07}.

\begin{theorem}[Regularizing properties of the heat flow]\label{thm:facili}
(Theorem~6.1 in {\rm \cite{Ambrosio-Gigli-Savare11b}}) Let $(X,\sfd,\mm)$ be a $RCD(K,\infty)$ space. Then:
\begin{itemize}
\item[(i)] The transition probability densities are symmetric
\begin{equation}\label{eq:symmetrictp}
\ked xt (y)=\ked yt (x)\qquad\text{$\mm\times\mm$-a.e.~in $X\times
X$, for all $t>0,$}
\end{equation}
and satisfy for all $x\in X$ the Chapman-Kolmogorov formula:
\begin{equation}
\label{eq:chapman} \ked x{t+s}(y)=\int\ked xt (z)\ked zs
(y)\,\d\mm(z)\qquad\text{for $\mm$-a.e.~$y\in X$, for all $t,\,s\geq
0$.}
\end{equation}
\item[(ii)] The formula
\begin{equation}
\label{eq:l1} \tilde{\sf h}_tf(x):=\int f(y)\,\d\ke xt(y)\qquad
x\in X
\end{equation}
provides a version of ${\sf h}_t f$ for every $f\in L^2(X,\mm)$, an
extension of ${\sf h}_t$ to a continuous contraction semigroup in
$L^1(X,\mm)$ which is pointwise everywhere defined if $f\in
L^\infty(X,\mm)$.
\item[(iii)] The semigroup $\tilde{\sf h}_t$ maps contractively $L^\infty(X,\mm)$ in
$C_b(X)$ and, in addition, $\tilde{\sf h}_tf(x)$ belongs to $
C_b\bigl((0,\infty)\times X\bigr)$.
\item[(iv)] If $f:X\to\R$ is Lipschitz, then $\tilde{\sf h}_tf$ is Lipschitz on $X$
as well and $\Lip(\tilde{\sf h}_tf)\leq e^{-Kt}\Lip(f)$.
\end{itemize}
\end{theorem}

\begin{theorem}[Bakry-\'Emery in $RCD(K,\infty)$ spaces]\label{thm:bakryemery}
(Theorem~6.2 in {\rm \cite{Ambrosio-Gigli-Savare11b}}) For any $f\in L^2(X,\mm)\cap\calW$ and $t>0$ we have
\begin{equation}\label{eq:bakryemery}
\weakgrad{({\sf h}_tf)}^2\leq
\rme^{-2Kt}{\sf h}_t(\weakgrad{f}^2)\qquad\text{$\mm$-a.e.~in $X$.}
\end{equation}
In addition, if
 $\weakgrad{f}\in L^\infty(X,\mm)$ and $t>0$, then
$\rme^{-Kt}\bigl(\tilde{\sf h}_t\weakgrad{f}^2\bigr)^{1/2}$ is an
upper gradient of $\tilde{\sf h}_tf$ on $X$, so that
\begin{equation}
\text{$|D^- \tilde{\sf h}_tf|\leq\rme^{-Kt}
\bigl(\tilde{\sf h}_t\weakgrad{f}^2\bigr)^{1/2}$\quad pointwise in
$X$,} \label{eq:17}
\end{equation}
and $f$ has a Lipschitz version $\tilde{f}:X\to\R$,
with ${\rm Lip}(\tilde{f})\leq\|\weakgrad{f}\|_\infty$.
\end{theorem}

The regularization properties \eqref{eq:24} of $EVI_K$-flows provide
an $L\log L$ regularization of the semigroup $\heatw_t$ 
starting from arbitrary measures in $\probt{X}$.
When $X$ is a $RCD(K,\infty)$-space with $K>0$, then 
combining the slope inequality for $K$-geodesically convex functionals
\cite[Lemma 2.4.13]{Ambrosio-Gigli-Savare08} 
\begin{displaymath}
	\ent\mu\le \frac 1{2K}|D^-\entv|^2(\mu)
\end{displaymath}
with the identity $|D^-\entv|^2(f\mm)=\int\weakgrad{f}^2/f\,\d\mm$ between slope and Fisher information,
 we get
the Logarithmic-Sobolev inequality 
\begin{equation}
	\label{eq:+logsob}
	\int_X f\log f\,\d\mm \le \frac 1{2K} \int_{f>0} \frac{\weakgrad f^2}f\,\d\mm
	\quad\text{if }\sqrt f\in W^{1,2}(X,\sfd,\mm),\ f\mm\in \prob X,
\end{equation} 
which in particular yields the hypercontractivity of $\heatl_t$,
see e.g.\ \cite{Aneetal00}. When $\heatl_t$ is ultracontractive,  i.e.\
there exists $p>1$ such that
\begin{equation}\label{eq:stronger}
\|{\sf h}_t f\|_p \leq C(t)\|f\|_1\qquad\forevery f\in
L^2(X,\mm),\,\,t>0,
\end{equation}
then one can also obtain global Lipschitz regularity
for the transition probabilities \cite[Proposition~6.4]{Ambrosio-Gigli-Savare11b}, see also \cite[Proposition~4.4]{GigliKuwadaOhta10}.
The stronger regularizing property \eqref{eq:stronger} is known to
be true, for instance, if doubling and Poincar\'e hold in
$(X,\sfd,\mm)$, see \cite[Corollary~4.2]{Sturm96}. 

We conclude this section with an example of application
of the Bakry-\'Emery estimate \eqref{thm:bakryemery}, which can be proven following
the $\Gamma$-calculus tools of Bakry \cite{Bakry06}, see
Theorem~6.5 in {\rm \cite{Ambrosio-Gigli-Savare11b}} for a detailed proof.

\begin{theorem}[Lipschitz regularization]\label{thm:lipreg}
If $f\in L^2(X,\mm)$ then ${\sf h}_t f\in \calW$ for
    every $t>0$ and
    \begin{equation}
      \label{eq:19}
      2\, \mathrm I_{2K}(t)\weakgrad {{\sf h}_t f }^2\le {\sf h}_t
      f^2\quad\mm\text{-a.e.\ in $X$};
    \end{equation}
    in particular,
  if $f\in L^\infty(X,\mm)$ then $\tilde{\sf h}_t f\in \Lip(X)$ for every
  $t>0$ with
  \begin{equation}
    \label{eq:28}
    \sqrt{2\,\mathrm I_{2K}(t)}\,{\rm Lip}(\tilde{\sf h}_t f)\le
    \|f\|_\infty\quad\forevery t>0.
  \end{equation}
\end{theorem}

\subsection{Connections with Dirichlet forms and Markov processes}

Since $\C$ is quadratic, lower semicontinuous in $L^2(X,\mm)$ and since
$\weakgrad{f}$ has strong locality properties, it turns out that the bilinear form
$\mathcal E$ associated to $\C$, whose domain is from now on restricted from $L^1(X,\mm)\cap\calW$ to
$L^2(X,\mm)\cap\calW$, is a local Dirichlet form. In the theory of
Dirichlet forms  a canonical object is the induced distance, namely 
\begin{equation}\label{eq:fuku2}
\sfd_{{\mathcal E}}(x,y):=\sup\left\{|\tilde g(x)-\tilde g(y)|:\
g\in D({\mathcal E}),\,\,[g]\leq\mm\right\}\qquad\forall (x,y)\in
X\times X,
\end{equation}
where the function $\tilde{g}$ is the continuous representative in the
Lebesgue class of $g$, see Theorem~\ref{thm:bakryemery}). Another canonical object
is the local energy measure, namely the measure $[u]$ defined by
$$
[u](\varphi):=\mathcal E(u,u\varphi)-\frac{1}{2}\mathcal E(u^2,\varphi)
\qquad\varphi\in L^2(X,\mm)\cap\calW.
$$
A consequence of Lemma~\ref{lem:vabbeserve} is that $[u]=\weakgrad{u}^2\mm$ for all $u\in L^2(X,\mm)\cap\calW$.
Also the distances can be identified:

\begin{theorem}[Identification of $\sfd_{{\mathcal E}}$ and $\sfd$]\label{thm:idistances}
(Theorem~6.10 of {\rm \cite{Ambrosio-Gigli-Savare11b}}) The function $\sfd_{{\mathcal E}}$ in \eqref{eq:fuku2} coincides
with $\sfd$ on $X \times X$.
\end{theorem}

Finally, using a tightness property of $\mathcal E$, the theory of Dirichlet forms can be applied
to obtain the representation of transition probabilities in terms of a continuous Markov process: 

\begin{theorem}[Brownian motion]\label{thm:brownian}
(Theorem~6.8 of {\rm \cite{Ambrosio-Gigli-Savare11b}}) Let $(X,\sfd,\mm)$ be a $RCD(K,\infty)$ space. There exists a unique
(in law) Markov process $\{{\mathbf X}_t\}_{\{t\geq 0\}}$ in
$(X,\sfd)$ with continuous sample paths in $[0,\infty)$ and
transition probabilities $\heatw_t(\delta_x)$, i.e.
\begin{equation}\label{eq:transitionmp}
{\mathbf P}\bigl({\mathbf X}_{s+t}\in A\bigl|{\mathbf
X}_s=x\bigr)=\ke xt(A) \qquad\forall s,\,t\geq 0,\,\,\text{$A$
Borel}
\end{equation}
for $\mm$-a.e. $x\in X$.
\end{theorem}

\subsection{Tensorization}

Recall that a metric space $(X,\sfd)$ is said to be non branching if the map $(\e_0,\e_t):\geo(X)\to X^2$
is injective for all $t\in (0,1)$, i.e., geodesics do not split.

\begin{theorem}[Tensorization]\label{thm:tensor}
(Theorem~6.13 of {\rm \cite{Ambrosio-Gigli-Savare11b}}) Let $(X,\sfd_X,\mm_X)$, $(Y,\sfd_Y,\mm_Y)$ be metric measure spaces and define the
product space $(Z,\sfd,\mm)$ as $Z:=X\times Y$,
$\mm:=\mm_X\times\mm_Y$ and
\[
\sfd\big((x,y),(x',y')\big):=\sqrt{\sfd_X^2(x,x')+\sfd_Y^2(y,y')}.
\]
Assume that both $(X,\sfd_X,\mm_X)$ and $(Y,\sfd_Y,\mm_Y)$ are $RCD(K,\infty)$ and non branching. 
Then $(Z,\sfd,\mm)$ is $RCD(K,\infty)$ and
non branching as well.
\end{theorem}

In \cite{AGSBaEm} the first two authors in collaboration with Savar\'e proved that the tensorization property of $RCD(K,\infty)$ 
persists even when the non branching assumption on the base spaces is removed.

\def\cprime{$'$}


\begin{thebibliography}{10}

 \bibitem{Ambrosio-Gigli11}
  {\sc L.~Ambrosio and N.~Gigli}, 
  {\em User's guide to optimal transport theory},
  To appear in the CIME Lecture Notes in Mathematics, B.Piccoli and F.Poupaud
  Eds.,  (2011).

 \bibitem{Ambrosio-Gigli-Savare08}
  {\sc L.~Ambrosio, N.~Gigli and G.~Savar\'e},
  {\em Gradient flows in metric spaces and in the space of probability measures}, 
  Lectures in Mathematics ETH Z\"urich, Birkh\"auser Verlag, Basel, second~ed., 2008.

 \bibitem{Ambrosio-Gigli-Savare11}
  \leavevmode\vrule height 2pt depth -1.6pt width 23pt,
  {\em Calculus and heat flows in metric measure spaces with {R}icci curvature bounded from below},
   arXiv:1106.2090, (2011), to appear on Invent. Math.

 \bibitem{Ambrosio-Gigli-Savare11b}
  \leavevmode\vrule height 2pt depth -1.6pt width 23pt,
  {\em Metric measure spaces with Riemannian Ricci curvature bounded from below},
  Submitted paper, arXiv:1109.0222,  (2011).

 \bibitem{AmbrosioGigliSavare12}
\leavevmode\vrule height 2pt depth -1.6pt width 23pt, {\em Density of Lipschitz functions and 
equivalence of weak gradients in metric measure spaces}, arXiv:1111.3730  (2011). 
to appear on Rev. Mat. Iberoamericana.

\bibitem{AGSBaEm}
\leavevmode\vrule height 2pt depth -1.6pt width 23pt,
  {\em Bakry-\'Emery curvature-dimension condition and Riemannian Ricci curvature bounds},
  Submitted paper, arXiv:1209.5786, (2012).

\bibitem{Aneetal00}
{\sc C.~An\'e, S.~Blach\`ere, D.~{Chafa\"\i}, P.~Foug\`eres, I.~Gentil,
  F.~Malreu, C.~Roberto, and G.~Scheffer}, {\em Sur les in\'egalit\'es de
  {S}obolev logarithmiques}, no.~10 in Panoramas et Synth\`eses, Soci\'et\'e
  Math\'ematique de France, 2000.

 %\bibitem{BS2010} 
 % {\sc K.~Bacher and K.T.~Sturm}
 % {\em Localization and Tensorization Properties of the Curvature-Dimension Condition for Metric Measure Spaces},
 % Journal Funct. Anal. \textbf{259} (2010), 28--56.

\bibitem{Bakry06}
{\sc D.~Bakry}, {\em Functional inequalities for {M}arkov semigroups}, in
  Probability measures on groups: recent directions and trends, Tata Inst.
  Fund. Res., Mumbai, 2006, 91--147.

% \bibitem{Braides}
%  {\sc A.~Braides},
%  {\em $\Gamma$-convergence for Beginners},
%  Oxford Lecture Series in Mathematics and its Applications, Vol. 22, Oxford University Press, 2002.

 \bibitem{Brezis}
  {\sc H.~Brezis},
  {\em Analyse Fonctionnelle. Th\'eorie et applications},
  Collection Math\'ematiques Appliqu\'ees pour la Ma\^itrise, Masson, Paris, 1983.

 \bibitem{Cheeger00}
  {\sc J.~Cheeger},
  {\em Differentiability of {L}ipschitz functions on metric measure spaces},
  Geom. Funct. Anal., \textbf{9} (1999), 428--517.

\bibitem{Cheeger-Colding97I}
{\sc J.~Cheeger and T.~Colding}, {\em On the structure of spaces with {R}icci
  curvature bounded below {I}}, J. Diff. Geom., 45 (1997), pp.~406--480.

\bibitem{Cheeger-Colding97II}
\leavevmode\vrule height 2pt depth -1.6pt width 23pt, {\em On the structure of
  spaces with {R}icci curvature bounded below {II}}, J. Diff. Geom., 54 (2000),
  pp.~13--35.


\bibitem{Cheeger-Colding97III}
\leavevmode\vrule height 2pt depth -1.6pt width 23pt, {\em On the structure of
  spaces with {R}icci curvature bounded below {III}}, J. Diff. Geom., 54
  (2000), pp.~37--74.


\bibitem{DalMaso}
  {\sc G.~Dal Maso},
  {\em An Introduction to $\Gamma$-convergence},
  Birkh\"auser, Boston, 1993.
  
 \bibitem{Daneri-Savare08}
{\sc S.~Daneri and G.~Savar\'e}, {\em {E}ulerian calculus for the displacement
  convexity in the {W}asserstein distance}, SIAM J. Math. Anal., \textbf{40} (2008),
  1104--1122.


\bibitem{Gigli10}
{\sc N.~Gigli}, {\em On the heat flow on metric measure spaces: existence,
  uniqueness and stability}, Calc. Var. PDE, 39 (2010), pp.~101--120.

  \bibitem{Gigli12}
  \leavevmode\vrule height 2pt depth -1.6pt width 23pt, {\em On the differential structure of metric measure spaces and
  applications}, Submitted paper, arXiv:1205.6622, (2012).

  \bibitem{GigliKuwadaOhta10}
{\sc N.~Gigli, K.~Kuwada, and S.~Ohta}, {\em Heat flow on {A}lexandrov spaces},
  Comm. Pure Appl. Math. \textbf{66} (2013), 307--331.

\bibitem{AmbrosioGigliMondinoSavare}
{\sc N.~Gigli, A.~Mondino and G.~Savar\'e},
{\em A notion of pointed convergence of non-compact metric measure spaces and stability of Ricci curvature bounds and heat flows.}
Preprint, (2012).

\bibitem{Gromov07}
{\sc M.~Gromov}, {\em Metric structures for {R}iemannian and non-{R}iemannian
  spaces}, Modern Birkh\"auser Classics, Birkh\"auser Boston Inc., Boston, MA,
  english~ed., 2007.
\newblock Based on the 1981 French original, With appendices by M. Katz, P.
  Pansu and S. Semmes, Translated from the French by Sean Michael Bates.

\bibitem{Heinonen-Koskela98}
{\sc J.~Heinonen and P.~Koskela}, {\em Quasiconformal maps in metric spaces
  with controlled geometry}, Acta Math., \textbf{181} (1998), 1--61.

\bibitem{Joulin07}
{\sc A.~Joulin}, {\em A new {P}oisson-type deviation inequality for {M}arkov
  jump processes with positive {W}asserstein curvature}, Bernoulli, \textbf{15} (2009),
  532--549.

\bibitem{Koskela-MacManus}
{\sc P.~Koskela and P.~MacManus}, {\em Quasiconformal mappings and {S}obolev
  spaces,} Studia Math., \textbf{131} (1998), 1--17.

\bibitem{Lisini07}
{\sc S.~Lisini}, {\em Characterization of absolutely continuous curves in
 {W}asserstein spaces}, Calc. Var. Partial Differential Equations, \textbf{28} (2007),
85--120.

\bibitem{Lott-Villani09}
{\sc J.~Lott and C.~Villani}, {\em Ricci curvature for metric-measure spaces
  via optimal transport}, Ann. of Math., \textbf{169} (2009), 903--991.

\bibitem{Ollivier09}
{\sc Y.~Ollivier}, {\em Ricci curvature of {M}arkov chains on metric measure
  spaces}, J. Functional Analysis, \textbf{256} (2009), 810--864.

 \bibitem{R2012}
  {\sc T.~Rajala},
  {\em Improved geodesics for the reduced curvature-dimension condition in branching metric spaces},
  Discrete Contin. Dyn. Syst., \textbf{33} (2013), 3043--3056.

 \bibitem{R2011b}
  \leavevmode\vrule height 2pt depth -1.6pt width 23pt,
  {\em Interpolated measures with bounded density in metric spaces satisfying the curvature-dimension conditions of Sturm},
  Journal Funct. Anal., \textbf{263} (2012),  896--924.

 \bibitem{R2011}
  \leavevmode\vrule height 2pt depth -1.6pt width 23pt,
  {\em Local Poincar\'e inequalities from stable curvature conditions on metric spaces},
  Calc. Var. Partial Differential Equations, \textbf{44} (2012), 477--494.

\bibitem{Shanmugalingam00}
{\sc N.~Shanmugalingam}, {\em Newtonian spaces: an extension of {S}obolev
  spaces to metric measure spaces}, Rev. Mat. Iberoamericana, \textbf{16} (2000),
  243--279.
  
\bibitem{Sturm96}
{\sc K.-T. Sturm}, {\em Analysis on local {D}irichlet spaces. {III}. {T}he
  parabolic {H}arnack inequality}, J. Math. Pures Appl., \textbf{75} (1996),
 273--297. 

\bibitem{Sturm06I}
\leavevmode\vrule height 2pt depth -1.6pt width 23pt, {\em On the geometry of
  metric measure spaces. {I}}, Acta Math., \textbf{196} (2006), 65--131.

\bibitem{Sturm06II}
\leavevmode\vrule height 2pt depth -1.6pt width 23pt, {\em On the geometry of
  metric measure spaces. {II}}, Acta Math., \textbf{196} (2006), 133--177.

\bibitem{Villani09}
{\sc C.~Villani}, {\em Optimal transport. Old and new}, vol.~338 of Grundlehren
  der Mathematischen Wissenschaften, Springer-Verlag, Berlin, 2009.

\end{thebibliography}
\end{document}